\documentclass[11pt]{amsart}

\usepackage[T1]{fontenc}
\usepackage{lmodern}
\usepackage{microtype}
\usepackage[a4paper,margin=30mm]{geometry}
\usepackage{mathtools}
\usepackage{amssymb,amsmath,amsthm,amscd}
\usepackage{mathrsfs}
\usepackage{enumitem}
\usepackage{aliascnt}
\usepackage[numbers,sort&compress]{natbib}
\usepackage[hypertexnames=false,hidelinks]{hyperref}

\setlist[enumerate]{leftmargin=2.25em}

\newtheorem{theorem}{Theorem}[section]
\newaliascnt{proposition}{theorem}
\newtheorem{proposition}[proposition]{Proposition}
\aliascntresetthe{proposition}
\newaliascnt{lemma}{theorem}
\newtheorem{lemma}[lemma]{Lemma}
\aliascntresetthe{lemma}
\newaliascnt{corollary}{theorem}
\newtheorem{corollary}[corollary]{Corollary}
\aliascntresetthe{corollary}
\theoremstyle{remark}
\newaliascnt{remark}{theorem}
\newtheorem{remark}[remark]{Remark}
\aliascntresetthe{remark}

\usepackage[capitalise,noabbrev]{cleveref}

\crefname{theorem}{Theorem}{Theorems}
\crefname{proposition}{Proposition}{Propositions}
\crefname{lemma}{Lemma}{Lemmas}
\crefname{corollary}{Corollary}{Corollaries}
\crefname{remark}{Remark}{Remarks}
\Crefname{theorem}{Theorem}{Theorems}
\Crefname{proposition}{Proposition}{Propositions}
\Crefname{lemma}{Lemma}{Lemmas}
\Crefname{corollary}{Corollary}{Corollaries}
\Crefname{remark}{Remark}{Remarks}

\numberwithin{equation}{section}

\newcommand{\F}{\mathbf F}
\newcommand{\Ga}{\mathbf G_{\mathrm a}}
\newcommand{\Gm}{\mathbf G_{\mathrm m}}
\newcommand{\GL}{\operatorname{GL}}
\newcommand{\Aut}{\operatorname{Aut}}
\newcommand{\Der}{\operatorname{Der}}
\newcommand{\Dist}{\operatorname{Dist}}
\newcommand{\Lie}{\operatorname{Lie}}
\newcommand{\Sym}{\operatorname{Sym}}
\newcommand{\im}{\operatorname{im}}
\newcommand{\gr}{\operatorname{gr}}
\newcommand{\Cal}{\operatorname{Cal}}
\newcommand{\id}{\operatorname{id}}
\newcommand{\red}{\mathrm{red}}
\newcommand{\alg}{\mathrm{alg}}
\newcommand{\dR}{\mathrm{dR}}
\newcommand{\ad}{\operatorname{ad}}
\newcommand{\Ad}{\operatorname{Ad}}
\newcommand{\Sq}{\operatorname{Sq}}
\newcommand{\cH}{\mathcal H}
\newcommand{\cE}{\mathcal E}
\newcommand{\cL}{\mathcal L}
\newcommand{\cW}{\mathcal W}
\newcommand{\cR}{\mathscr R}
\newcommand{\fm}{\mathfrak m}
\newcommand{\abs}[1]{\lvert#1\rvert}

\title[Extremal Derived Length and Unipotent Radicals]{Extremal Derived Length\\
in the Unipotent Radicals of Cartan-Type Automorphism Groups}

\author{Chao Ma}
\address{Independent Researcher, London, United Kingdom}
\email{raymond.ma@me.com}
\urladdr{https://orcid.org/0009-0004-2456-9098}
\date{}


\hypersetup{
  pdftitle={Extremal Derived Length in the Unipotent Radicals of Cartan-Type Automorphism Groups},
  pdfauthor={Chao Ma},
  pdfcreator={Chao Ma},
  pdfsubject={Unipotent radicals of the automorphism group schemes of the restricted Cartan-type Lie algebras, and the lower central and derived series they carry},
  pdfkeywords={derived length, nilpotency class, lower central series, derived series, unipotent radical, automorphism group scheme, Cartan-type Lie algebra, special Lie algebra, Witt-Jacobson algebra, Hamiltonian Lie algebra, contact Lie algebra, congruence filtration, Cartier operator, Verschiebung, Frobenius thickening, modular Lie algebra, finite p-group}
}

\begin{document}

\begin{abstract}
Over perfect fields of characteristic \(p\ge5\), the unipotent radicals of
the height-one special, Witt--Jacobson, and Hamiltonian automorphism groups
in at least three, two, and four variables, respectively, attain the
largest derived length allowed by their nilpotency classes.  More strongly, every term of each derived series equals the
corresponding power-of-two term of the lower central series.  These
equalities hold on finite-field points, with explicit Frattini quotients
and lower \(p\)-central series.  For the special family, a
Cartier quotient measures the difference between the congruence and lower
central filtrations.  The Hamiltonian correction comes from flux and the
socle line; the Witt--Jacobson family has none.  We also identify the contact automorphism group scheme and, at
arbitrary divided-power height, compute the Cartier--Verschiebung maps on
restricted derivations and the Verschiebung filtration and distribution
algebra of the Frobenius thickening of the special radical.
\end{abstract}

\maketitle
\section{Introduction}
\label{sec:introduction}

In any group one has \(G^{(j)}\subseteq\gamma_{2^j}(G)\), so a nilpotent
group of class \(c\) has derived length at most
\(\lceil\log_2(c+1)\rceil\).  The unipotent radical \(U_S\) of the automorphism group of the special algebra \(S(n;1)^{(1)}\) attains this bound at every step: by
\cref{thm:special-main,cor:extremal-derived-length}, with \(D=n(p-1)\),
\[
 U_S^{(j)}=\gamma_{2^j}(U_S)\quad(j\ge1),
 \qquad
 \operatorname{cl}(U_S)=D-2,
 \qquad
 \operatorname{dl}(U_S)=\lceil\log_2(D-1)\rceil,
\]
and the same holds for the finite groups \(U_S(\F_q)\), for which the
Frattini quotient and the lower \(p\)-central series are also obtained in
closed form.  The mechanism is a congruence filtration whose graded Lie
algebra is generated by brackets in all but finitely many degrees; a single
Cartier quotient governs the exceptional degrees.

The Hamiltonian family behaves in the same way, one step lower.  Let \(U_H\)
be the tangent-to-identity symplectic substitution group of
\(H(2r;\mathbf1)^{(2)}\), let \(K_H\) be the kernel of its flux character and
\(Z_M\) its socle line, and put \(D_H=2r(p-1)\).  For \(r\ge2\) we prove in
\cref{thm:hamiltonian-series,thm:hamiltonian-derived} that
\[
 F^{j+1}U_H\cap K_H=\Gamma_j(U_H)\times Z_M
 \quad(2\le j\le D_H-2),
 \qquad
 U_H^{(j)}=\Gamma_{2^j}(U_H)\quad(j\ge1),
\]
so that
\[
 \operatorname{cl}(U_H)=D_H-3,
 \qquad
 \operatorname{dl}(U_H)=\lceil\log_2(D_H-2)\rceil .
\]
The drop by one relative to the special case is exactly the socle line, which
is central and is not reached by brackets.
These series formulas also hold on finite-field points.  For \(q=p^f\),
\cref{thm:hamiltonian-finite} identifies the lower \(p\)-central series
with the lower central series and gives
\[
 U_H(\F_q)/\Phi(U_H(\F_q))
 \simeq (\F_q,+)^{\binom{2r+2}{3}+2r+1}.
\]
For \(r=2\) and \(p=5\), the nilpotency class is \(13\), the derived
length is \(4\), and the minimum number of generators is \(25f\).

The commutant of the positive part of a graded simple Lie algebra of Cartan
type was determined by Kostrikin and Shafarevich
\cite{KostrikinShafarevich1969}: in the restricted simple case the algebra is
generated by its components of degree \(-1\) and \(1\), and the positive part
is generated by its degree-one component.  We use these statements in the form
of \cite[Ch.~4, Theorems~7.2--7.5]{StradeFarnsteiner1988}.  The height-one
truncation departs from them in finitely many degrees.  A single Cartier
quotient governs these degrees, and together they determine the derived and
lower central series.

Let \(k\) be a perfect field of characteristic \(p\ge5\).  Two features of
automorphism group schemes of restricted Cartan-type Lie algebras are
specific to positive characteristic: the schemes need not be smooth, and
the reduced identity components can carry large unipotent radicals.  Both
are visible in the divided-power coordinate algebra.  The notation for
the four classical Cartan families follows \cite{Strade2004}; the relevant
coordinate realizations of their automorphism schemes are those of
\cite{BahturinKochetov2011}.

Waterhouse treated the Witt--Jacobson automorphism scheme and its forms
\cite{Waterhouse1971}.  Wilson determined the quotients of the canonical
filtration of graded Cartan-type automorphism groups
\cite[Theorems~1 and~2]{Wilson1975}.  Shu and Shen gave functorial coordinate
descriptions for the special, Hamiltonian, and contact families
\cite{ShuShen1995}.  Skryabin described Witt--Jacobson automorphism group
schemes over a base ring \cite{Skryabin2001}.  He also proved that
Cartan-type isomorphisms over commutative rings arise from admissible
coordinate-algebra isomorphisms and compatible coefficient modules carrying
the defining differential forms \cite{Skryabin1995}.  Premet, and later
Shen and Shu in the adjoint graded setting, studied the
algebraic groups generated by exponential one-parameter subgroups attached
to \(p\)-nilpotent elements \cite{Premet1985,ShenShu1998}.  Here we study
instead the internal structure of the tangent-to-identity unipotent kernels.

Write
\[
 O(n;1)=k[x_1,\ldots,x_n]/(x_1^p,\ldots,x_n^p),
 \qquad \fm=(x_1,\ldots,x_n),
\]
and let \(W(n;1)\), \(S(n;1)^{(1)}\), \(H(2r;1)^{(2)}\), and
\(K(2r+1;1)^{(1)}\) denote the simple restricted Lie algebras of Witt--Jacobson,
special, Hamiltonian, and contact type, with the usual derived modification
in the exceptional contact case.  Their coordinate automorphism groups are,
respectively, the full automorphism group of \(O(n;1)\) and the stabilizers
of the volume line, the symplectic line, and the contact line.

Each tangent-to-identity kernel carries the congruence filtration
\[
 F^sU=\{g\in U:g(x_i)-x_i\in\fm^s\ \text{for all }i\},\qquad s\ge2.
\]
This is the filtration of \cite[Theorem~1]{Wilson1975}, where its successive
quotients are identified and every element of a layer is written as
\(\id+\ad D+(\text{higher order})\) with \(D\) homogeneous.

\begin{remark}[Grading convention]
\label{rem:grading-convention}
Homogeneous components of a Cartan-type Lie algebra are indexed throughout by
the standard grading, in which \(x^a\partial_j\) has degree \(\abs a-1\); see
\cite{KostrikinShafarevich1969} and \cite[Ch.~4]{StradeFarnsteiner1988}.  A
substitution with \(g(x_i)-x_i\in\fm^s\) has tangent field with coefficients
in \(\fm^s\), hence Lie degree \(s-1\).  Accordingly the congruence subgroup
\(F^sU\) is assigned \emph{filtration degree \(s-1\)}, and the layer
\(\cL_d\) of coefficient degree \(d\) sits in Lie degree \(d-1\).  We write \(\gr^sU=F^sU/F^{s+1}U\) for the layer with coordinate label \(s\)
and \(\gr_eU\) for the layer in Lie degree \(e\), so that
\(\gr_eU=\gr^{e+1}U\).  With this normalization the
bracket and the \(p\)-map on the associated graded are additive
(\cref{prop:special-realization}).
\end{remark}  By
Kreknin \cite{Kreknin1971}, \(\Aut L=\Aut_0L\) for \(p\ge5\), so the
filtration exhausts the automorphism group.  Wilson's description
yields the inclusions \([F^aU,F^bU]\subseteq F^{a+b-1}U\).  We prove the corresponding equalities, and they determine the lower central
and derived series.
In the associated graded, commutators become Lie brackets and \(p\)-th
powers become restricted powers.  Exceptional quotients occur precisely in
the degrees where the truncated de Rham complex has cohomology.  There
bracket generation fails, and the classical Cartier operator describes the
quotients \cite{Cartier1957}.  Related calculations appear in
\cite{MaDerived2026} for the derived algebra of nonlinear divergence-free
polynomial fields, and in \cite{MaAugmentation2026} for the augmentation
filtration of special Nottingham congruence series.

For type \(S\), only one Cartier degree occurs.  Let
\(\chi_{\Box}:U_S\to H_{\dR}^{n-1}(O(n;1))\) be the Cartier character,
and set
\[
 K_s=F^sU_S\cap\ker\chi_{\Box}\qquad(s\ge3).
\]
Then
\[
 \gamma_j(U_S)=K_{j+1}\quad(j\ge2),\qquad
 U_S^{(j)}=K_{2^j+1}\quad(j\ge1),
\]
so that
\[
 \operatorname{cl}(U_S)=D-2,\qquad
 \operatorname{dl}(U_S)=\left\lceil\log_2(D-1)\right\rceil.
\]
Over \(\F_q\), the same filtration gives
\[
 U_S(\F_q)/\Phi(U_S(\F_q))
 \simeq
 \cL_2(\F_q)\oplus
 \bigl((\det V^*)^{p-1}\otimes V^{(1)}\bigr)(\F_q)
\]
and
\[
P_i(U_S(\F_q))=F^{i+1}U_S(\F_q)\cap\ker\chi_{\Box}
 \qquad(i\ge2).
\]
The Cartier summand also lies in the image of the five-term homology
sequence (\cref{thm:cartier-relation-map}).

For the Hamiltonian family, homogeneous Poisson brackets generate every
component except the box socle.  When \(r+1\not\equiv0\pmod p\), an algebraic
Calabi homomorphism detects the socle line.  At a resonance
\[
 r+1=hp,
\]
the Calabi homomorphism vanishes on the socle.  Symplectic-equivariant
linearization of the lower-central quotients and a calculation of
low-degree symplectic covariants handle this case.  For every \(r\ge2\),
the socle line has trivial reduced intersection with the algebraic derived
subgroup.

A canonical Frobenius thickening \(\mathbf U_S^F\) of \(U_S\) carries the
non-smooth structure in type \(S\); we compute its Verschiebung layers and
its distribution algebra.  At arbitrary divided-power
height, compatible Cartier--Verschiebung morphisms on the de Rham complexes
induce maps on restricted derivations.  For the remaining height-one families we determine the
Witt--Jacobson and Hamiltonian reduced radicals, and identify the full contact
automorphism group scheme with the coordinate contact stabilizer.

\section{Coordinate groups and filtrations}
\label{sec:setup}

Throughout, \(k\) is a perfect field of characteristic \(p\ge5\).  For a
finite extension \(\F_q/k_0\), where \(k_0=\F_p\), all coordinate groups and
filtrations are formed over \(k_0\) and then evaluated on \(\F_q\).
The affine group-scheme and distribution conventions are those of
\cite{Waterhouse1979}.

\subsection{The height-one divided-power algebra}

Write
\[
 O(n;1)=k[x_1,\ldots,x_n]/(x_1^p,\ldots,x_n^p),
 \qquad
 \omega_S=dx_1\wedge\cdots\wedge dx_n.
\]
The maximal ideal is \(\fm=(x_1,\ldots,x_n)\), and the top box degree is
\[
 D=n(p-1).
\]
The Witt--Jacobson algebra is \(W(n;1)=\Der_k O(n;1)\).  Contraction with
\(\omega_S\) identifies divergence-free fields with closed
\((n-1)\)-forms; a divergence-free field is \emph{exact} when its
contraction is an exact form.
Write \(\cL_d\) for the degree-\(d\) divergence-free component.

The natural adjoint morphisms identify the full automorphism schemes of
\(W(n;1)\), \(S(n;1)^{(1)}\), and \(H(2r;1)^{(2)}\) with the coordinate
automorphism scheme, the volume-line stabilizer, and the symplectic-line
stabilizer, respectively
\cite[Theorems~3.1, 3.2, and~3.5]{BahturinKochetov2011}.  Since the adjoint
morphisms and their inverse coordinate constructions are defined over the
prime field, these identifications descend from an algebraic closure to
every perfect ground field.
The Witt--Jacobson comparison originates in \cite{Waterhouse1971}.

Let
\[
 G_S=\left\{g\in\Aut_k O(n;1):g^*(k\omega_S)=k\omega_S\right\}_{\red}.
\]
The linear part gives a split epimorphism \(G_S\to\GL_n\); its kernel is
\[
 U_S=\left\{g\in\Aut_k O(n;1):
 g(x_i)\equiv x_i\pmod{\fm^2},\ g^*\omega_S=\omega_S\right\}.
\]
For \(s\ge2\), set
\[
 F^sU_S=\left\{g\in U_S:g(x_i)-x_i\in\fm^s\ \forall i\right\}.
\]

For the Witt--Jacobson family write
\[
 G_W=(\mathbf{Aut}(O(n;1)))_{\red},
 \qquad U_W=\ker(G_W\to\GL_n),
\]
where the map records the linear part.

\subsection{Hamiltonian coordinates}

For the Hamiltonian family put
\[
 O_H=O(2r;1)
 =k[q_1,\ldots,q_r,p_1,\ldots,p_r]/
 (q_1^p,\ldots,q_r^p,p_1^p,\ldots,p_r^p)
\]
and
\[
 \omega_H=\sum_{i=1}^r dq_i\wedge dp_i.
\]
For \(f\in O_H\), define
\[
 X_f=\sum_{i=1}^r
 \left(
 \frac{\partial f}{\partial p_i}\frac{\partial}{\partial q_i}
 -
 \frac{\partial f}{\partial q_i}\frac{\partial}{\partial p_i}
 \right),
 \qquad
 \{f,g\}=X_f(g).
\]
Then \([X_f,X_g]=X_{\{f,g\}}\).

Let
\[
 G_H=
 \operatorname{Stab}_{\Aut(O_H)}(k\omega_H)_{\red}.
\]
The linear-part morphism has image \(\operatorname{CSp}_{2r}\); denote its
kernel, the reduced tangent-to-identity symplectic substitution group, by
\(U_H\).  With
\[
 \theta_H=\frac12\sum_{i=1}^r(q_i\,dp_i-p_i\,dq_i),
 \qquad d\theta_H=\omega_H,
\]
the flux character is
\[
 \chi_H(g)=[g^*\theta_H-\theta_H]\in H^1_{\dR}(O_H).
\]
The tangent-to-identity action on \(H^1_{\dR}(O_H)\) is trivial, hence
\(\chi_H\) is a group homomorphism.  Write
\[
 K_H=\ker\chi_H .
\]

For \(s\ge2\), put
\[
 F^sU_H=\{g\in U_H:g(x)-x\in\fm^s\text{ for every linear }x\}.
\]

The box socle and its Hamiltonian field are
\[
 M=\prod_{i=1}^r q_i^{p-1}p_i^{p-1},
 \qquad
 X_M.
\]
Put
\[
 D_H=2r(p-1),\qquad N=D_H-2.
\]
Scalar dilation acts on \(X_M\) with weight \(N\).  The corresponding
socle one-parameter subgroup is denoted \(Z_M\simeq\Ga\).

Write \(\cH_d\) for the coefficient-degree-\(d\) component of symplectic
fields and \(\cE_d=X_{O_{H,d+1}}\) for its Hamiltonian-exact subspace.
Contraction with \(\omega_H\) gives
\[
 \cH_d/\cE_d\simeq
 \begin{cases}
 H^1_{\dR}(O_H),&d=p-1,\\
 0,&d\ne p-1.
 \end{cases}
\]

\subsection{Algebraic and abstract commutators}

For a smooth algebraic group \(G\), write \([G,G]_{\alg}\) for the
subgroup generated by commutators.  If \(A,B\subseteq G\) are closed and
connected, the subgroup generated by the commutators \((a,b)\) is already
closed and connected \cite[I,~2.3]{Borel1991},
\cite[\S17.1]{Humphreys1975}, so the abstract and algebraic commutator subgroups of
closed connected subgroups agree.  The algebraic lower-central series is
\[
 \Gamma_1(G)=G,\qquad
 \Gamma_{j+1}(G)=[G,\Gamma_j(G)]_{\alg}.
\]
For a finite-point group \(G(\F_q)\), brackets denote abstract commutators.
For finite \(p\)-groups, set
\[
 \Phi(P)=P^p[P,P],\qquad
 P_1(P)=P,\qquad
 P_{i+1}(P)=P_i(P)^p[P_i(P),P]
\]
for the Frattini and lower \(p\)-central series of a finite \(p\)-group.

\begin{lemma}[Frobenius decomposition]
\label{lem:frobenius-decomposition}
Let \(E_1,\ldots,E_m,F\) be vector groups with linear actions of a
torus \(T\).  Every \(T\)-equivariant multi-homomorphism
\[
 \Psi:E_1\times\cdots\times E_m\longrightarrow F
\]
is a finite sum of components homogeneous of degree \(p^{e_i}\)
in the \(i\)-th variable.  A component from characters
\(\lambda_1,\ldots,\lambda_m\) to a target character \(\lambda\) can be
nonzero only when
\[
 \lambda=\sum_{i=1}^m p^{e_i}\lambda_i.
\]
\end{lemma}

\begin{proof}
Choose linear coordinates on the vector groups.  Each coordinate function
of \(\Psi\) is additive in every block of variables.  If a polynomial
\(f(X_1,\ldots,X_s)\) is additive, comparison of coefficients in
\(f(X+Y)=f(X)+f(Y)\) shows that every exponent \(a\) which occurs satisfies
\(\binom{a}{b}=0\) in \(k\) for \(0<b<a\); hence \(a\) is a power of
\(p\).  Applying the same argument successively to the \(m\) blocks yields the
stated decomposition.  Torus equivariance on each monomial component then
forces the displayed character identity.
\end{proof}
\section{Main theorems}
\label{sec:main-theorems}

\begin{theorem}[Special automorphism groups]
\label{thm:special-main}
Let \(n\ge3\).  The reduced identity component of the automorphism scheme of
\(S(n;1)^{(1)}\) is
\[
 \Aut(S(n;1)^{(1)})_{\red}^{\circ}
 \simeq U_S\rtimes\GL_n,
\]
and \(U_S\) is its unipotent radical.

Let \(D=n(p-1)\), let
\[
 \chi_{\Box}:U_S\longrightarrow
 H^{n-1}_{\dR}(O(n;1))
 \simeq(\det V^*)^{p-1}\otimes V^{(1)}
\]
be the Cartier character, and put
\[
 K_s=F^sU_S\cap\ker\chi_{\Box}\qquad(s\ge3).
\]
Then
\[
 [K_s,U_S]=K_{s+1},\qquad [K_s,K_s]=K_{2s-1},
\]
\[
 \gamma_j(U_S)=K_{j+1}\quad(j\ge2),\qquad
 U_S^{(j)}=K_{2^j+1}\quad(j\ge1).
\]
In particular,
\[
 \operatorname{cl}(U_S)=D-2,\qquad
 \operatorname{dl}(U_S)=\left\lceil\log_2(D-1)\right\rceil.
\]

For every \(q=p^f\),
\[
 [U_S(\F_q),U_S(\F_q)]
 =\Phi(U_S(\F_q))
 =\ker(\,j_2,\chi_{\Box}\,)(\F_q),
\]
\[
 U_S(\F_q)/\Phi(U_S(\F_q))
 \simeq
 \cL_2(\F_q)\oplus
 \bigl((\det V^*)^{p-1}\otimes V^{(1)}\bigr)(\F_q),
\]
and
\[
 P_i(U_S(\F_q))
 =F^{i+1}U_S(\F_q)\cap\ker\chi_{\Box}
 \qquad(i\ge2).
\]
In particular, the minimal number of abstract generators of
\(U_S(\F_q)\) is
\[
 f\left(\frac{n(n-1)(n+2)}2+n\right).
\]
\end{theorem}

\begin{corollary}[Extremal derived length]
\label{cor:extremal-derived-length}
For every \(j\ge1\),
\[
 U_S^{(j)}=\gamma_{2^j}(U_S).
\]
Thus the inclusion \(G^{(j)}\subseteq\gamma_{2^j}(G)\), valid in every
group, is an equality at every step for \(U_S\).  Consequently \(U_S\) has
the largest derived length its nilpotency class permits: a nilpotent group
of class \(c\) has derived length at most \(\lceil\log_2(c+1)\rceil\), and
here \(c=D-2\) while \(\operatorname{dl}(U_S)=\lceil\log_2(D-1)\rceil\).

The same holds for the finite groups \(U_S(\F_q)\): for every \(q=p^f\),
\[
 U_S(\F_q)^{(j)}=\gamma_{2^j}\bigl(U_S(\F_q)\bigr)
 \qquad(j\ge1),
\]
with the same class and derived length.
\end{corollary}

\begin{proof}
By \cref{thm:special-main}, \(U_S^{(j)}=K_{2^j+1}\) for \(j\ge1\) and
\(\gamma_i(U_S)=K_{i+1}\) for \(i\ge2\); taking \(i=2^j\) gives the
displayed equality.  The inclusion \(G^{(j)}\subseteq\gamma_{2^j}(G)\) holds in every group, and if \(\gamma_{c+1}(G)=1\) then \(G^{(j)}=1\) as soon as
\(2^j\ge c+1\).

Both recurrences of \cref{prop:special-commutators} are identities of the
congruence filtration and are formed over \(\F_q\) exactly as in the proof
of \cref{prop:special-frattini}.  They give
\(\gamma_i(U_S(\F_q))=K_{i+1}(\F_q)\) and
\(U_S(\F_q)^{(j)}=K_{2^j+1}(\F_q)\), whence the finite-point statement.
\end{proof}

\begin{theorem}[Cartier classes in the five-term sequence]
\label{thm:cartier-relation-map}
Put
\[
 B_q=\cL_2(\F_q),\qquad
 C_q=H^{n-1}_{\dR}(O(n;1))(\F_q),\qquad
 A_q=B_q\oplus C_q.
\]
The canonical summand
\[
 \cR_{\Box,q}^{\mathrm{Car}}
 =(B_q\otimes_{\F_p}C_q)\oplus H_2(C_q;\F_p)
\]
is contained in the image of
\[
 H_2(U_S(\F_q);\F_p)\longrightarrow H_2(A_q;\F_p).
\]
The \(H_2(C_q;\F_p)\)-summand splits.
\end{theorem}

\begin{theorem}[Canonical Frobenius thickening]
\label{thm:frobenius-thickening}
Let \(\mathbf G_S\) be the full automorphism group scheme of
\(S(n;1)^{(1)}\), and let
\[
 \mathbf U_S^F
 =\ker\!\left(
 \mathbf G_S\xrightarrow{\overline F}G_S^{(p)}
 \longrightarrow\GL_n^{(p)}
 \right).
\]
Then
\[
 (\mathbf U_S^F)_{\red}=U_S
\]
and there is an exact sequence
\[
 1\longrightarrow(\mathbf G_S)_1
 \longrightarrow\mathbf U_S^F
 \xrightarrow{\overline F}U_S^{(p)}
 \longrightarrow1.
\]
Moreover
\[
 \Lie_0(\mathbf U_S^F)=C_S,\qquad
 \Lie_j(\mathbf U_S^F)=S(n;1)_{(1)}\quad(j\ge1),
\]
and multiplication induces
\[
 u(C_S)\otimes_{u(S(n;1)_{(1)})}\Dist(U_S)
 \xrightarrow{\sim}\Dist(\mathbf U_S^F).
\]
\end{theorem}

\begin{theorem}[The Hamiltonian socle line]
\label{thm:hamiltonian-main}
Let \(r\ge1\), let \(U_H\) be the reduced height-one tangent-to-identity
symplectic substitution group, and let \(Z_M\) be its socle Hamiltonian
line.

\begin{enumerate}
\item If \(r=1\), then \(Z_M\subseteq[U_H,U_H]_{\alg}\).
\item If \(r\ge2\), then
\[
 \bigl(Z_M\cap[U_H,U_H]_{\alg}\bigr)_{\red}=1.
\]
\item For every \(q=p^f\) and \(r\ge2\),
\[
 Z_M(\F_q)\cap[U_H(\F_q),U_H(\F_q)]=1.
\]
\end{enumerate}

\end{theorem}

\begin{theorem}[Hamiltonian finite-point series]
\label{thm:hamiltonian-finite}
Let \(r\ge2\), \(q=p^f\), and \(U_{H,q}=U_H(\F_q)\).
For every \(j\ge1\) and \(a\ge0\),
\[
 \gamma_j(U_{H,q})=\Gamma_j(U_H)(\F_q),\qquad
 P_j(U_{H,q})=\gamma_j(U_{H,q}),\qquad
 U_{H,q}^{(a)}=\gamma_{2^a}(U_{H,q}).
\]
The Frattini subgroup and quotient satisfy
\[
 \Phi(U_{H,q})=\Gamma_2(U_H)(\F_q),\qquad
 U_{H,q}/\Phi(U_{H,q})
 \simeq(\F_q,+)^{\binom{2r+2}{3}+2r+1}.
\]
Consequently the minimal number of generators of \(U_{H,q}\) is
\[
 f\left(\binom{2r+2}{3}+2r+1\right),
\]
and
\[
 \operatorname{cl}(U_{H,q})=D_H-3,\qquad
 \operatorname{dl}(U_{H,q})=\lceil\log_2(D_H-2)\rceil.
\]
\end{theorem}

\begin{theorem}[Witt--Jacobson, Hamiltonian, and contact automorphism groups]
\label{thm:witt-contact-main}
For the height-one Witt--Jacobson group with \(n\ge2\),
\[
 G_W=U_W\rtimes\GL_n,
 \qquad R_u(G_W)=U_W,
\]
\[
 [U_W,U_W]_{\alg}=F^3U_W,\qquad
 \gamma_j(U_W)=F^{j+1}U_W,
\]
and, for every \(s\ge2\),
\[
 [F^sU_W,F^sU_W]_{\alg}=F^{2s-1}U_W.
\]
Moreover, with \(D=n(p-1)\), for every \(j\ge0\),
\[
 U_W^{(j)}=F^{2^j+1}U_W=\gamma_{2^j}(U_W).
\]
Consequently
\[
 \operatorname{cl}(U_W)=D-1,\qquad
 \operatorname{dl}(U_W)=\left\lceil\log_2D\right\rceil.
\]
For every \(q=p^f\),
\[
 \Phi(U_W(\F_q))=F^3U_W(\F_q),\qquad
 P_j(U_W(\F_q))=F^{j+1}U_W(\F_q).
\]
The same self-commutator and derived-series equalities hold on
\(U_W(\F_q)\).

For the height-one Hamiltonian group,
\[
 G_H=U_H\rtimes\operatorname{CSp}_{2r},
 \qquad R_u(G_H)=U_H.
\]

For the simple height-one contact algebra, including the exceptional
codimension-one derived case, the adjoint morphism from the coordinate
contact group scheme to the full automorphism group scheme is an isomorphism.
\end{theorem}
\section{The special group}
\label{sec:special-group}

Write \(O=O(n;1)\), \(D=n(p-1)\), and
\[
 d_0=(p-1)(n-1).
\]
Contraction with \(\omega_S\) identifies the homogeneous component
\(\cL_d\) of divergence-free fields with closed \((n-1)\)-forms of coefficient degree \(d\).  Let
\[
 E_d=\ker\!\left(
 \cL_d\longrightarrow H^{n-1}_{\dR}(O)
 \right)
\]
be its exact subspace.

\subsection{The reduced radical and its graded Lie algebra}

\begin{proposition}
\label{prop:special-radical}
The linear-part map gives a split exact sequence
\[
 1\longrightarrow U_S\longrightarrow G_S
 \longrightarrow\GL_n\longrightarrow1,
\]
and
\[
 U_S=R_u(G_S).
\]
\end{proposition}

\begin{proof}
Linear substitution by \(A\in\GL_n\) sends \(\omega_S\) to
\(\det(A)\omega_S\), providing a section of the linear-part map.  If
\(g\) lies in the identity-linear kernel, preservation of the line
\(k\omega_S\) and comparison of constant terms in the Jacobian imply
\(g^*\omega_S=\omega_S\).

Order the monomial basis of \(O\) by total degree.  An identity-linear
substitution is upper unitriangular in this basis, so \(U_S\) is unipotent.
Scalar dilation contracts it:
\[
 \delta_tg\delta_t^{-1}(x_i)
 =x_i+\sum_{\abs a\ge2}t^{\abs a-1}c_{i,a}x^a.
\]
The action extends regularly to \(t=0\) with value the identity, so \(U_S\)
is connected.  The quotient \(G_S/U_S\simeq\GL_n\) is reductive, so every
connected normal unipotent subgroup of \(G_S\) is contained in \(U_S\).
\end{proof}

\begin{proposition}[Associated graded of the congruence filtration]
\label{prop:special-realization}
For \(2\le d\le D-1\), the leading-field map induces
\[
 F^dU_S/F^{d+1}U_S\simeq\cL_d.
\]
In the Lie degrees of \cref{rem:grading-convention} this reads
\(\gr_{d-1}U_S\simeq\cL_d\), and the associated graded is a graded Lie
algebra:
\[
 [\,\gr_aU_S,\gr_bU_S\,]\subseteq\gr_{a+b}U_S
\]
is the ordinary Lie bracket of vector fields, while the first possible symbol
of a \(p\)-th power from degree \(a\) lies in degree \(pa\) and equals the
restricted power.  In the coordinate labels the same two statements read
\([F^aU_S,F^bU_S]\subseteq F^{a+b-1}U_S\) and depth \(1+p(a-1)\).
\end{proposition}

\begin{proof}
If
\[
 g(x_i)=x_i+f_i+\text{terms of degree \(>d\)},\qquad \deg f_i=d,
\]
then the first homogeneous term in \(\det J_g-1\) is
\(\sum_i\partial_if_i\).  Hence the leading field
\(\sum_if_i\partial_i\) is divergence-free.

For surjectivity, let \(X\) be a divergence-free homogeneous field.  The
substitution \(x\mapsto x+X(x)\) preserves the volume form through its
leading degree.  Suppose inductively that a partial lift preserves
\(\omega_S\) modulo coefficient degree \(m\).  Its error is a closed top
form \(\varepsilon_m\omega_S\).  The zero integral forced by the
change-of-variables formula puts this form in the image of
\[
 d:\Omega^{n-1}_{m+1}(O)\longrightarrow\Omega^n_m(O).
\]
Choose a vector field \(Y_m\) whose contraction with \(\omega_S\) is a
primitive, and compose with \(x\mapsto x-Y_m(x)\).  The composition removes
the error without changing lower-degree terms.  Induction through the
finite congruence filtration produces an element of \(U_S\) with leading
field \(X\).

The commutator formula follows by expanding two substitutions to the first
nonzero common degree.  For powers, apply the substitution difference
operator \(g-1\) repeatedly.  In characteristic \(p\), the intermediate
binomial coefficients vanish, and the first possible nonzero term is the
\(p\)-fold iterate of the leading derivation, at depth
\(1+p(a-1)\).  Its symbol is the restricted power.
\end{proof}

\subsection{The Cartier character}

The tensor-product decomposition of the truncated de Rham complex gives
\[
 H^\bullet_{\dR}(O)
 \simeq
 \Lambda_k\bigl(
 [x_1^{p-1}dx_1],\ldots,[x_n^{p-1}dx_n]
 \bigr).
\]
In degree \(n-1\) this has basis
\[
 \eta_i=
 \left[
 \left(\prod_{j\ne i}x_j^{p-1}\right)
 dx_1\wedge\cdots\wedge\widehat{dx_i}\wedge\cdots\wedge dx_n
 \right].
\]
All these classes have coefficient degree \(d_0\).

Choose
\[
 \theta=x_1\,dx_2\wedge\cdots\wedge dx_n,
 \qquad d\theta=\omega_S.
\]

\begin{proposition}
\label{prop:special-cartier}
The formula
\[
 \chi_{\Box}(g)=[g^*\theta-\theta]
\]
defines a surjective homomorphism
\[
 \chi_{\Box}:U_S\longrightarrow H^{n-1}_{\dR}(O)
 \simeq(\det V^*)^{p-1}\otimes V^{(1)}.
\]
Its associated-graded map is zero outside degree \(d_0\), and
\[
 \ker(\gr^{d_0}\chi_{\Box})=E_{d_0}.
\]
\end{proposition}

\begin{proof}
The form \(g^*\theta-\theta\) is closed because \(g\) preserves
\(\omega_S\).  A tangent-to-identity substitution acts trivially on the
displayed Cartier basis: nonlinear terms raise coefficient degree, while
this cohomology group has support only in degree \(d_0\).  Hence
\[
 \chi_{\Box}(gh)=h^*\chi_{\Box}(g)+\chi_{\Box}(h)
 =\chi_{\Box}(g)+\chi_{\Box}(h).
\]

Put \(m_i=\prod_{j\ne i}x_j^{p-1}\).  The substitutions
\[
 s_i(a):\quad x_i\mapsto x_i+am_i,\qquad x_j\mapsto x_j\ (j\ne i)
\]
have inverse \(s_i(-a)\) and Jacobian one.  A direct pullback calculation
gives
\[
 \chi_{\Box}(s_i(a))=(-1)^{i-1}a\eta_i.
\]
The shears realize the Cartier basis, so \(\chi_{\Box}\) is surjective.

For a leading field \(X\), Cartan's formula gives
\[
 \mathcal L_X\theta
 =\iota_X\omega_S+d(\iota_X\theta).
\]
Thus the induced graded character is
\(X\mapsto[\iota_X\omega_S]\), whose kernel is the exact component.
\end{proof}

\subsection{Exact brackets and abelianization}

\begin{lemma}[Quadratic generation of exact components]
\label{lem:special-exact-generation}
For \(3\le d\le D-1\),
\[
 [\cL_2,E_{d-1}]=E_d.
\]
\end{lemma}

\begin{proof}
For divergence-free fields \(X,Y\), Cartan's formula gives
\[
 \iota_{[X,Y]}\omega_S=d(\iota_X\iota_Y\omega_S),
\]
so \([\cL_2,E_{d-1}]\subseteq E_d\).  It therefore suffices to generate the
monomial Jacobian fields
\[
 D_{ij}(h)=\partial_jh\,\partial_i-\partial_ih\,\partial_j,
 \qquad \deg h=d+1,
\]
which span \(E_d\).

Fix such a monomial, rename the distinguished variables \(x,y\), and write
\[
 h=x^Ay^BN,
\]
with \(N\) independent of \(x,y\).  If \(a,b\ge0\), \(a+b=3\), and the
indicated exponents are nonnegative, then
\begin{equation}
\label{eq:special-plane-bracket}
\begin{aligned}
 &[D_{xy}(x^ay^b),
 D_{xy}(x^{A-a+1}y^{B-b+1}N)]\\
 &\hspace{24mm}
 =\bigl(b(A+1)-a(B+1)\bigr)D_{xy}(h).
\end{aligned}
\end{equation}
The first factor lies in \(\cL_2\), and the second in \(E_{d-1}\).

Assume first that \(A,B\ge1\).  The choices \((a,b)=(2,1)\) and
\((1,2)\) are both admissible.  Their coefficients vanish simultaneously
only when
\[
 A\equiv B\equiv-1\pmod p,
\]
hence, in the height-one box, only when \(A=B=p-1\).  Hence
\eqref{eq:special-plane-bracket} generates the target unless
\(A=B=p-1\).  If \(A=0\) and \(B\ge2\), the choice \((a,b)=(0,3)\)
has coefficient \(3\); the case \(B=0,A\ge2\) is symmetric.  When
\(A=B=0\), the Jacobian field is zero.

Only two boundary configurations remain.  First
let \((A,B)=(0,1)\), so \(h=yN\).  Since \(\deg N=d\ge3\), choose a
coordinate \(z\) occurring in \(N\), say \(N=z^CM_0\) with \(C>0\).  Put
\[
 U=z^2\partial_x,
 \qquad
 V=D_{yz}\!\left(\frac12yz^{C-1}M_0\right).
\]
Then \(U\in\cL_2\), \(V\in E_{d-1}\), and
\[
 V=-\frac12z^{C-1}M_0\partial_z
   +\frac{C-1}{2}yz^{C-2}M_0\partial_y,
 \qquad [U,V]=N\partial_x=D_{xy}(h),
\]
with the second summand omitted when \(C=1\).  The configuration
\((A,B)=(1,0)\) is symmetric.

It remains to take \(A=B=p-1\).  If \(N\ne1\), choose a coordinate
\(z\) in its support and write \(N=z^CM\), where \(C>0\) and \(M\) is
independent of \(z\).  Set
\[
 \begin{aligned}
 U_1&=z^2\partial_x,&
 V_1&=D_{yz}(x^Ay^Bz^{C-1}M),\\
 U_2&=yz\partial_x,&
 V_2&=D_{yz}(x^Ay^{B-1}z^CM).
 \end{aligned}
\]
All four fields lie in the required source spaces, and, since \(B=-1\) in
\(k\),
\[
 [U_2,V_2]-2[U_1,V_1]=(C-2)D_{xy}(h).
\]
If \(C\ne2\), the coefficient is nonzero.  If \(C=2\), put
\[
 U_3=z^2\partial_y,
 \qquad V_3=D_{xz}(x^Ay^BzM),
\]
and obtain
\[
 D_{xy}(h)=[U_1,V_1]-[U_3,V_3].
\]

If \(N=1\), choose any third coordinate \(z\) and define
\[
 R=x^Ay^{B-2}\partial_z
   =\frac12D_{yz}(x^Ay^{B-1}),
\]
\[
 T=D_{yz}(x^Ay^{B-2}z)
  =x^Ay^{B-2}\partial_y
   +3x^Ay^{B-3}z\partial_z.
\]
Then \(R,T\in E_{d-1}\) and
\[
 D_{xy}(h)=3[zy\partial_x,R]-[y^2\partial_x,T].
\]
All denominators used above are invertible because \(p\ge5\).  Hence every
monomial generator of \(E_d\) lies in \([\cL_2,E_{d-1}]\).
\end{proof}

Let
\[
 j_2:U_S\longrightarrow\cL_2
\]
be the quadratic leading jet.

\begin{proposition}
\label{prop:special-abelianization}
The map
\[
 \psi=(j_2,\chi_{\Box}):
 U_S\longrightarrow
 \cL_2\oplus H^{n-1}_{\dR}(O)
\]
is a surjective homomorphism with kernel
\([U_S,U_S]_{\alg}\).
\end{proposition}

\begin{proof}
Both components are homomorphisms.  To prove joint surjectivity, fix
\((X,c)\) in the target and choose \(g\in U_S\) with \(j_2(g)=X\), using
\cref{prop:special-realization}.  The Cartier shears from
\cref{prop:special-cartier} lie in \(F^{d_0}U_S\subseteq F^3U_S\), so they
have zero quadratic jet and their \(\chi_{\Box}\)-values span
\(H^{n-1}_{\dR}(O)\).  Hence there is such a shear product \(s\) with
\[
 \chi_{\Box}(s)=c-\chi_{\Box}(g),
\]
and then \(\psi(gs)=(X,c)\).

Since the target is commutative, \([U_S,U_S]_{\alg}\subseteq\ker\psi\).
For the reverse inclusion, let \(g\in\ker\psi\), \(g\ne1\), and let \(d\)
be its leading degree.  The condition \(j_2(g)=0\) gives \(d\ge3\).  If
\(d\ne d_0\), the degree-\(d\) de Rham cohomology vanishes, so the leading
field \(X\in\cL_d\) lies in \(E_d\).  If \(d=d_0\), the equality
\(\chi_{\Box}(g)=0\) and the associated-graded description in
\cref{prop:special-cartier} give the same conclusion.

By \cref{lem:special-exact-generation}, write
\[
 X=\sum_{\nu}[A_\nu,B_\nu],
 \qquad A_\nu\in\cL_2,
 \quad B_\nu\in E_{d-1}.
\]
Choose lifts \(a_\nu\in U_S\) and \(b_\nu\in F^{d-1}U_S\) with these
leading fields.  The product
\[
 c=\prod_\nu[a_\nu,b_\nu]
\]
has leading field \(X\), belongs to \([U_S,U_S]_{\alg}\), and therefore
lies in \(\ker\psi\).  Hence \(c^{-1}g\in F^{d+1}U_S\cap\ker\psi\).  Iteration through the
finite congruence filtration places \(g\) in \([U_S,U_S]_{\alg}\).
\end{proof}

For \(s\ge3\), put
\[
 K_s=F^sU_S\cap\ker\chi_{\Box}.
\]
Since \(F^3U_S=\ker j_2\), the preceding proposition gives
\begin{equation}
\label{eq:special-K3-derived}
 K_3=\ker(j_2,\chi_{\Box})=[U_S,U_S]_{\alg}.
\end{equation}

\begin{lemma}[Balanced brackets]
\label{lem:balanced-special}
If \(s\ge3\) and \(2s\le D\), then
\[
 [E_s,E_s]=E_{2s-1}.
\]
\end{lemma}

\begin{proof}
It suffices to treat \(D_{ij}(h)\) for a monomial \(h=x^a\) of
degree \(2s\).  If \(a_i,a_j\le s\), factor
\[
 h=fg,\qquad
 \deg f=\deg g=s,\qquad
 \partial_if=0,\quad\partial_jg=0.
\]
Assign all \(x_j\)-power to \(f\), all \(x_i\)-power to \(g\), and split
the remaining powers to reach degree \(s\).  Then
\[
 [f\partial_i,g\partial_j]=-D_{ij}(h).
\]

Suppose, after interchanging \(i,j\), that \(h=x_i^Ar\) with \(A>s\).
Let \(b\) be the exponent of \(x_j\) in \(r\), and put
\[
 X=x_i^s\partial_j,
\]
\[
 Y=x_i^{A-s}r\,\partial_i
 -\frac{A-s}{b+1}x_i^{A-s-1}x_jr\,\partial_j.
\]
Since \(A\le p-1\) and \(A>s\), one has \(s\le p-2\); hence
\(b\le\deg r<s\) and \(b+1<p\).  The denominator is therefore invertible, and all exponents remain in the
box.  Direct calculation gives
\[
 \operatorname{div}X=\operatorname{div}Y=0,\qquad
 [X,Y]=D_{ij}(h).
\]
Since \(s<d_0\), one has \(\cL_s=E_s\), so the preceding brackets lie in
\([E_s,E_s]\).  Exactness of brackets gives the reverse inclusion.
\end{proof}

\begin{proposition}
\label{prop:special-commutators}
For every \(s\ge3\),
\[
 [K_s,U_S]=K_{s+1},\qquad
 [K_s,K_s]=K_{2s-1}.
\]
\end{proposition}

\begin{proof}
By \cref{prop:special-cartier}, the degree-\(d\) component of \(K_s\) is
\(E_d\) for \(d\ge s\); when \(d<d_0\), a nonzero Cartier value of a
lift is removed by a degree-\(d_0\) Cartier shear without changing its
leading field.  For \(d\ge s+1\), \cref{lem:special-exact-generation} gives
\[
 E_d=[\cL_2,E_{d-1}],
\]
so the leading component of every element of \(K_{s+1}\) is realized by a
commutator from \([U_S,K_s]\).  Removing that component and iterating over
the finite congruence filtration yields
\[
 [K_s,U_S]=K_{s+1}.
\]

For the self-commutator, \cref{lem:balanced-special} gives
\(E_{2s-1}=[E_s,E_s]\).  Suppose the bracket span contains all components
through degree \(d\).  Since \(E_{d+1}=[\cL_2,E_d]\), Jacobi gives
\[
 [\cL_2,[E_a,E_b]]
 \subseteq[[\cL_2,E_a],E_b]+[E_a,[\cL_2,E_b]],
\]
and the right-hand side again consists of brackets of components of degree
at least \(s\).  Hence every \(E_d\), \(d\ge2s-1\), occurs in the graded self-commutator.
The same leading-term elimination gives
\([K_s,K_s]=K_{2s-1}\).  If \(2s>D\), both groups are trivial by depth.
\end{proof}

\begin{proof}[Proof of the lower-central and derived assertions in
\cref{thm:special-main}]
By \eqref{eq:special-K3-derived}, \([U_S,U_S]_{\alg}=K_3\).  Hence
\cref{prop:special-commutators} gives
\(\gamma_j(U_S)=K_{j+1}\), while its second recurrence gives
\[
 U_S^{(j)}=K_{2^j+1}.
\]
The highest nonzero exact component is \(E_{D-1}\), and the divergence-free
component of coefficient degree \(D\) vanishes.  Hence the
nilpotency class is \(D-2\) and the derived length is
\(\lceil\log_2(D-1)\rceil\).
\end{proof}

\subsection{Finite points and the lower
\texorpdfstring{\(p\)}{p}-central series}

\begin{lemma}[Power depth]
\label{lem:power-depth}
For every \(q=p^f\) and \(a\ge2\),
\[
 (F^aU_S(\F_q))^p\subseteq F^{1+p(a-1)}U_S(\F_q).
\]
\end{lemma}

\begin{proof}
Write \(g=\id+\Delta\) on the coordinate algebra.  The operator \(\Delta\)
raises \(\fm\)-adic degree by \(a-1\).  In the expansion of \(g^p\), the
terms of operator degree \(1,\ldots,p-1\) occur in cyclic \(p\)-orbits and
have zero total coefficient in characteristic \(p\).  Every surviving term contains at least \(p\) degree-raising factors.
Hence \(g^p-\id\) raises degree by at least \(p(a-1)\), as required.
\end{proof}

\begin{proposition}
\label{prop:special-frattini}
For \(q=p^f\),
\[
 [U_S(\F_q),U_S(\F_q)]
 =\Phi(U_S(\F_q))
 =\ker(j_2,\chi_{\Box})(\F_q),
\]
and, for \(i\ge2\),
\[
 P_i(U_S(\F_q))
 =F^{i+1}U_S(\F_q)\cap\ker\chi_{\Box}.
\]
\end{proposition}

\begin{proof}
The filtration argument in \cref{prop:special-abelianization} is defined
over \(\F_q\), and gives the same commutator equality on finite points.
The target of \((j_2,\chi_{\Box})\) has exponent \(p\), so every
\(p\)-th power lies in the commutator subgroup.  Hence the commutator
subgroup is the Frattini subgroup.

Set \(H_i=F^{i+1}U_S(\F_q)\cap\ker\chi_{\Box}\).  The commutator equality
gives \(P_2=H_2\).  The first recurrence in
\cref{prop:special-commutators}, formed over \(\F_q\), gives
\([H_i,U_S(\F_q)]=H_{i+1}\).  By \cref{lem:power-depth},
\[
 H_i^p\subseteq F^{1+pi}U_S(\F_q)\cap\ker\chi_{\Box}
 \subseteq H_{i+1}.
\]
Induction yields
\[
 P_{i+1}=P_i^p[P_i,U_S(\F_q)]=H_{i+1}.
\]
\end{proof}

\begin{proof}[Proof of \cref{thm:special-main}]
The radical, commutator, and finite-point formulas are
\cref{prop:special-radical,prop:special-commutators,%
prop:special-abelianization,prop:special-frattini}.  The Cartier quotient is
\cref{prop:special-cartier}.  Finally,
\[
 \dim_{\F_q}\cL_2=\frac{n(n-1)(n+2)}2,
\]
and the Frattini quotient has one additional \(n\)-dimensional Cartier
summand over \(\F_q\), giving the stated generator number.
\end{proof}

\subsection{The Frattini extension and its transgression}

Put
\[
 B_q=\cL_2(\F_q),\qquad
 C_q=H^{n-1}_{\dR}(O)(\F_q),\qquad
 A_q=B_q\oplus C_q.
\]
The five-term sequence for
\[
 1\longrightarrow P_2\longrightarrow U_S(\F_q)
 \longrightarrow A_q\longrightarrow1
\]
contains
\[
 H_2(U_S(\F_q);k_0)\longrightarrow H_2(A_q;k_0)
 \xrightarrow{\operatorname{tr}_q}P_2/P_3.
\]
The homological normalization is that of \cite[Chapter~VII]{Brown1982}.

\begin{proposition}
\label{prop:special-transgression}
Under
\[
 H_2(A_q;k_0)\simeq
 H_2(B_q;k_0)\oplus(B_q\otimes C_q)\oplus H_2(C_q;k_0),
\]
the transgression is zero on the last two summands.  On
\[
 H_2(B_q;k_0)\simeq
 \Lambda^2B_q\oplus\Gamma_1^{\mathrm{dp}}(B_q)
\]
it is
\[
 \beta_q\oplus0,
\qquad
 \beta_q(b\wedge b')=[b,b']\in\cL_3(\F_q),
\]
and \(\beta_q\) is surjective.
\end{proposition}

\begin{proof}
The Cartier shears \(s_i(a)\) commute: their commutator depth is
\(2d_0-1>D\).  They have order \(p\) and give a group-theoretic section
\(C_q\to U_S(\F_q)\).  Naturality makes the transgression zero on
\(H_2(C_q;k_0)\).  A mixed class is represented by the commutator of a
quadratic lift and a Cartier shear.  Its depth is at least
\(d_0+1\), hence it belongs to \(P_3\); the mixed transgression is zero.

For the quadratic part, the normalized bar representatives show that the
exterior coordinate is the commutator of two lifts and the divided-power
coordinate is the \(p\)-th power of one lift.  Modulo \(P_3\), the former
has leading field \([b,b']\).  By \cref{lem:power-depth}, the latter lies
in \(F^{p+1}U_S(\F_q)\).  Moreover
\(\chi_{\Box}(g^p)=p\chi_{\Box}(g)=0\), since the Cartier target has
exponent \(p\).  Hence
\[
 (F^2U_S(\F_q))^p
 \subseteq F^{p+1}U_S(\F_q)\cap\ker\chi_{\Box}
 \subseteq F^4U_S(\F_q)\cap\ker\chi_{\Box}=P_3.
\]
Finally, \([\cL_2,\cL_2]=\cL_3\) is the degree-three instance of
\cref{lem:special-exact-generation}, so \(\beta_q\) is onto.
\end{proof}

\begin{proof}[Proof of \cref{thm:cartier-relation-map}]
By \cref{prop:special-transgression}, the summand
\[
 (B_q\otimes C_q)\oplus H_2(C_q;k_0)
\]
lies in the kernel of the transgression and therefore in the image of
\(H_2(U_S(\F_q);k_0)\).  The section \(C_q\to U_S(\F_q)\) splits the
\(H_2(C_q;k_0)\)-summand.
\end{proof}
\section{The canonical infinitesimal thickening}
\label{sec:infinitesimal}

Let
\[
 \mathbf G_S
 =\operatorname{Stab}_{\mathbf{Aut}(O)}(k\omega_S)
\]
be the full automorphism group scheme.  Its reduced subgroup is
\(G_S=U_S\rtimes\GL_n\).

\subsection{Relative Frobenius}

\begin{lemma}
\label{lem:frobenius-factorization}
Relative Frobenius factors through the augmentation-preserving reduced
group:
\[
 \overline F:\mathbf G_S\longrightarrow G_S^{(p)}.
\]
It is finite faithfully flat, and
\[
 \ker\overline F=(\mathbf G_S)_1.
\]
\end{lemma}

\begin{proof}
Write a universal coordinate automorphism as
\[
 \varphi(x_i)=a_i+\sum_{0<\abs\alpha}c_{i,\alpha}x^\alpha.
\]
The relation \(\varphi(x_i)^p=0\) gives \(a_i^p=0\).  On coordinate Hopf
algebras, relative Frobenius therefore sends every target constant
coordinate to zero, so its image preserves the augmentation.

For a test algebra \(R\), let \(\tau_a(x_i)=x_i+a_i\).
Every \(\varphi\in\mathbf G_S(R)\) factors uniquely as
\[
 \varphi=\psi\circ\tau_a,
\qquad
 \psi(0)=0,\quad a\in\alpha_p^n(R).
\]
Translations preserve \(\omega_S\), and the construction is compatible
with base change.

The augmentation-preserving factor is the smooth group \(G_S\).  To verify
smoothness directly, let \(R\to R/I\) be a square-zero
extension and lift an augmentation-preserving volume-line substitution
from \(R/I\) arbitrarily, retaining its invertible linear part.  Its
Jacobian error is an \(I\)-valued top form.  With
\(\operatorname{res}(f\omega_S)=[x_1^{p-1}\cdots x_n^{p-1}]f\), change of
variables gives
\(\operatorname{res}(\varphi^*(f\omega_S))=\operatorname{res}(f\omega_S)\).
Thus \(\operatorname{res}(\det J_\varphi)=0\).  After the constant part is
absorbed into the volume multiplier, the error has zero residue, so the
truncated de Rham calculation writes each homogeneous error as
\[
 d\beta,\qquad \beta\in\Omega^{n-1}(O)\otimes I.
\]
If \(Y\) is defined by \(\iota_Y\omega_S=-\beta\), composition with
\(\id+Y\) removes that homogeneous error without changing lower terms.
Induction over the finite congruence filtration removes all homogeneous
errors.  The zero-constant subfunctor is therefore smooth and coincides with
its reduced subgroup \(G_S\).  Multiplication then gives an
isomorphism of schemes
\[
 G_S\times\alpha_p^n\xrightarrow{\sim}\mathbf G_S.
\]
Under this decomposition, \(\overline F\) is relative Frobenius on the
smooth factor and is trivial on \(\alpha_p^n\).
\end{proof}

Define the fibre product
\[
 \mathbf U_S^F
 =\mathbf G_S\times_{G_S^{(p)}}U_S^{(p)}.
\]

\begin{proposition}
\label{prop:frobenius-thickening}
There are exact sequences
\[
 1\longrightarrow(\mathbf G_S)_1
 \longrightarrow\mathbf U_S^F
 \xrightarrow{\overline F}U_S^{(p)}
 \longrightarrow1
\]
and, for \(r\ge1\),
\[
 1\longrightarrow(\mathbf G_S)_1
 \longrightarrow(\mathbf U_S^F)_r
 \longrightarrow(U_S)_{r-1}^{(p)}
 \longrightarrow1.
\]
Moreover
\[
 (\mathbf U_S^F)_{\red}=U_S.
\]
\end{proposition}

\begin{proof}
The first exact sequence is the base change of
\cref{lem:frobenius-factorization} along
\(U_S^{(p)}\hookrightarrow G_S^{(p)}\).  Fppf-locally,
\[
 \mathbf U_S^F=(\mathbf G_S)_1U_S,
 \qquad
 (\mathbf G_S)_1\cap U_S=(U_S)_1.
\]
Taking Frobenius kernels gives the second sequence.  Since
\((\mathbf G_S)_1\) is infinitesimal, the reduced subgroup is \(U_S\).
\end{proof}

\subsection{The first primitive layer}

Put
\[
 C_S=\{D\in W(n;1):\operatorname{div}D\in k\}.
\]

\begin{proposition}
\label{prop:primitive-layers}
\[
 \Lie_0(\mathbf U_S^F)=C_S,\qquad
 \Lie_j(\mathbf U_S^F)=S(n;1)_{(1)}\quad(j\ge1).
\]
In particular,
\[
 \Lie_0(\mathbf U_S^F)/\Lie_1(\mathbf U_S^F)
 \simeq V\oplus\mathfrak{gl}_n.
\]
\end{proposition}

\begin{proof}
The tangent algebra of the full special automorphism scheme is \(C_S\).
Verschiebung is the distribution map induced by \(\overline F\).  A
primitive in the first positive layer must map to
\[
 \Lie(U_S^{(p)})=S(n;1)_{(1)}^{(p)};
\]
under the Frobenius-twist identification this places it in
\(S(n;1)_{(1)}\).  The reverse inclusion follows from the smooth subgroup
\(U_S\), whose tangent vectors possess divided-power sequences of every
length.  Since the Verschiebung filtration is descending, the same
equality holds for all \(j\ge1\).

The graded decomposition
\[
 C_S
 =\operatorname{span}\{\partial_1,\ldots,\partial_n\}
 \oplus\mathfrak{gl}_n\oplus S(n;1)_{(1)}
\]
gives the quotient.
\end{proof}

\begin{proposition}
\label{prop:no-unipotent-thickening}
There is no normal unipotent subgroup scheme
\(\mathbf N\triangleleft\mathbf G_S\) with
\(\mathbf N_{\red}=U_S\).
\end{proposition}

\begin{proof}
If such \(\mathbf N\) existed, \(\Lie(\mathbf N)\) would be an ideal in
\(C_S\) containing \(S(n;1)_{(1)}\).  Choose \(i\ne j\).  Since
\(x_i^2\partial_j\in S(n;1)_{(1)}\), ideality gives
\[
 [\partial_i,x_i^2\partial_j]=2x_i\partial_j
 \in\Lie(\mathbf N).
\]
Bracketing with \(x_j\partial_i\in C_S\) gives
\[
 x_j\partial_j-x_i\partial_i\in\Lie(\mathbf N).
\]
The resulting element is toral:
\[
 (x_j\partial_j-x_i\partial_i)^{[p]}
 =x_j\partial_j-x_i\partial_i.
\]
The restricted Lie algebra of a connected unipotent group scheme has no
nonzero toral element: such a primitive generates
\(k[t]/(t^p-t)=\Dist(\mu_p)\) in its first Frobenius kernel, a
contradiction.
\end{proof}

\subsection{Distribution factorization}

\begin{proposition}
\label{prop:distribution-factorization}
Multiplication induces an isomorphism of filtered vector spaces
\[
 u(C_S)\otimes_{u(S(n;1)_{(1)})}\Dist(U_S)
 \xrightarrow{\sim}\Dist(\mathbf U_S^F).
\]
\end{proposition}

\begin{proof}
For every \(r\ge1\),
\[
 (\mathbf U_S^F)_r=(\mathbf G_S)_1(U_S)_r,
 \qquad
 (\mathbf G_S)_1\cap(U_S)_r=(U_S)_1.
\]
Applying distributions gives a surjective multiplication map
\[
 u(C_S)\otimes_{u(S(n;1)_{(1)})}\Dist((U_S)_r)
 \longrightarrow\Dist((\mathbf U_S^F)_r).
\]
Both sides have dimension
\[
 p^{\dim C_S+(r-1)\dim S(n;1)_{(1)}
 }.
\]
Equality of dimensions makes the map an isomorphism.  Taking the union
over \(r\) gives the stated factorization.
\end{proof}

\begin{proof}[Proof of \cref{thm:frobenius-thickening}]
The exact sequence and reduced subgroup are given by
\cref{prop:frobenius-thickening}; the Verschiebung layers and distribution
algebra are \cref{prop:primitive-layers,prop:distribution-factorization},
respectively.
\end{proof}
\section{Cartier--Verschiebung maps at arbitrary height}
\label{sec:varying-height}

Let
\[
 O(\mathbf m)
 =O(n;m_1,\ldots,m_n)
\]
be the divided-power algebra with basis
\[
 x^{(\alpha)}=\prod_i x_i^{(\alpha_i)},
 \qquad 0\le\alpha_i<p^{m_i}.
\]
Multiplication is
\[
 x^{(\alpha)}x^{(\beta)}
 =\binom{\alpha+\beta}{\alpha}x^{(\alpha+\beta)}
\]
when the exponents remain in the box.  Let \(\Omega^\bullet(\mathbf m)\)
be its de Rham complex.

\subsection{Finite-height cohomology}

\begin{proposition}
\label{prop:finite-height-de-rham}
The de Rham cohomology is the exterior algebra on the \(n\) classes
\[
 \xi_{i,m_i}
 =x_i^{(p^{m_i}-1)}dx_i.
\]
In particular,
\[
 H^s_{\dR}(O(\mathbf m))
 \simeq
 \bigoplus_{\abs I=s}
 k\left[
 \prod_{i\in I}x_i^{(p^{m_i}-1)}
 \bigwedge_{i\in I}dx_i
 \right].
\]
\end{proposition}

\begin{proof}
In one variable, the divided-power derivative is
\[
 d(x^{(a)})=x^{(a-1)}dx.
\]
It maps the span of \(x^{(1)},\ldots,x^{(p^m-1)}\) isomorphically onto
the span of \(x^{(0)}dx,\ldots,x^{(p^m-2)}dx\).  Hence the one-variable
cohomology is
\[
 H^0=k,\qquad
 H^1=k[x^{(p^m-1)}dx].
\]
The multivariable complex is their tensor product, and the K\"unneth
formula gives the result.
\end{proof}

\subsection{Height-raising morphisms}

For \(J\subseteq\{1,\ldots,n\}\), let \(\mathbf m+\mathbf 1_J\) denote the
height vector obtained by raising the coordinates in \(J\).  Define
\[
 p_J\alpha_i=
 \begin{cases}
 p\alpha_i,&i\in J,\\
 \alpha_i,&i\notin J,
 \end{cases}
\]
and
\[
 \mathfrak C_J(x^{(\alpha)})=x^{(p_J\alpha)}.
\]
On one-forms put
\[
 \mathfrak C_J(dx_i)=
 \begin{cases}
 x_i^{(p-1)}dx_i,&i\in J,\\
 dx_i,&i\notin J,
 \end{cases}
\]
and extend multiplicatively.

\begin{theorem}
\label{thm:cartier-verschiebung}
The map
\[
 \mathfrak C_J:
 \Omega^\bullet(\mathbf m)
 \longrightarrow
 \Omega^\bullet(\mathbf m+\mathbf 1_J)
\]
is an injective morphism of differential graded algebras.  The maps satisfy
\[
 \mathfrak C_I\mathfrak C_J=\mathfrak C_{I\sqcup J}
\]
when \(I\cap J=\varnothing\); arbitrary iterated height increments are
obtained by composing these coordinate maps.  On cohomology,
\[
 \left[
 \prod_{i\in I}x_i^{(p^{m_i}-1)}
 \bigwedge_{i\in I}dx_i
 \right]
 \longmapsto
 \left[
 \prod_{i\in I}x_i^{(p^{m_i+\mathbf1_{i\in J}}-1)}
 \bigwedge_{i\in I}dx_i
 \right].
\]
In particular, every induced cohomology map is an isomorphism.
\end{theorem}

\begin{proof}
Lucas' theorem gives
\[
 \binom{p(a+b)}{pa}\equiv\binom{a+b}{a}\pmod p,
\]
so \(x^{(a)}\mapsto x^{(pa)}\) preserves divided-power multiplication.
The coordinatewise tensor product gives multiplicativity of
\(\mathfrak C_J\).  Injectivity follows from the divided-power basis.

Let \(\partial_i\) be the divided-power derivative and let
\(\mathcal V_J\) be the function part of \(\mathfrak C_J\).  For
\(i\in J\),
\[
 \partial_i\mathcal V_J(f)
 =x_i^{(p-1)}\partial_i^p\mathcal V_J(f)
 =x_i^{(p-1)}\mathcal V_J(\partial_i f).
\]
For \(i\notin J\), one has
\(\partial_i\mathcal V_J=\mathcal V_J\partial_i\).  Together these
identities express the compatibility of \(\mathfrak C_J\) with \(d\).
Composition follows on basis elements.  After \(e\) repetitions in a
raised coordinate,
\[
 x^{(a)}\longmapsto x^{(p^ea)},
 \qquad dx\longmapsto x^{(p^e-1)}dx.
\]

Finally,
\[
 x_i^{(p^{m_i}-1)}dx_i
 \longmapsto
 x_i^{(p(p^{m_i}-1))}x_i^{(p-1)}dx_i
 =x_i^{(p^{m_i+1}-1)}dx_i.
\]
The last equality uses
\(\binom{p^{m_i+1}-1}{p-1}\equiv1\pmod p\), again by Lucas' theorem.
The cohomology statement follows from
\cref{prop:finite-height-de-rham}.
\end{proof}

\subsection{Restricted derivations and Hasse operators}

Let \(\mathcal V_J\) be the function part of \(\mathfrak C_J\), and put
\[
 \mu_J=\prod_{i\in J}x_i^{(p-1)}.
\]
In a raised coordinate the ordinary derivative on the image factors as
\[
 \partial_i
 =M_{x_i^{(p-1)}}\partial_i^p,
 \qquad
 \partial_i^p\mathcal V_J=\mathcal V_J\partial_i.
\]
Let \(\widetilde{\mathcal V}_J\) denote the map on coefficient modules
induced by \(\mathfrak C_J\) under contraction with the volume forms.  Then
\[
\begin{CD}
 O(\mathbf m)\otimes k^n @>{\sum_i\partial_i}>> O(\mathbf m)\\
 @V{\widetilde{\mathcal V}_J}VV
 @VV{M_{\mu_J}\mathcal V_J}V\\
 O(\mathbf m+\mathbf1_J)\otimes k^n
 @>{\sum_i\partial_i}>> O(\mathbf m+\mathbf1_J)
\end{CD}
\]
is commutative.

At equal height, write
\[
 c_{i,m}=
 \left[
 \prod_{j\ne i}x_j^{(p^m-1)}
 \bigwedge_{j\ne i}dx_j
 \right].
\]
Then
\[
 \mathfrak C_{\{1,\ldots,n\}}(c_{i,m})=c_{i,m+1}.
\]
On the corresponding coefficient space define
\[
 w_{i,m}=\left(\prod_{j\ne i}x_j^{(p^m-1)}\right)e_i
 \qquad(m\ge1).
\]
The transition is the Hasse operator
\[
 \mathscr H(w_{i,m})
 =\left(\prod_{j\ne i}x_j^{(p-1)}\right)
  \mathcal V_{\{1,\ldots,n\}}
  \!\left(\prod_{j\ne i}x_j^{(p^m-1)}\right)e_i
 =w_{i,m+1}.
\]

\begin{corollary}
\label{cor:hasse-chains}
In the equal-height directed system, the Hasse operator \(\mathscr H\)
links the classes \(c_{i,m}\), \(m\ge1\), into \(n\) chains, one for each
omitted coordinate \(i\).  Under contraction
with the volume form, their derivative component is the restricted power
\(\partial_i^{[p]}=\partial_i^p\), followed by multiplication by the
Hasse factor \(x_i^{(p-1)}\).
\end{corollary}

\begin{proof}
The cohomology basis in \cref{prop:finite-height-de-rham} consists of the
\(n\) classes \(c_{i,m}\) in degree \(n-1\).  After contraction with the
volume form, the differential identity in
\cref{thm:cartier-verschiebung} becomes the stated restricted-derivative
factorization.
\end{proof}

\begin{remark}
The induced volume multiplier is
\[
 \prod_{i\in J}x_i^{(p-1)}.
\]
Thus the Cartier--Verschiebung map commutes with the de Rham
differential and multiplies the distinguished special volume form by this
factor.
\end{remark}
\section{Hamiltonian flux and homogeneous Poisson brackets}
\label{sec:hamiltonian-structure}

\subsection{The flux character}

\begin{proposition}
\label{prop:hamiltonian-flux}
The map
\[
 \chi_H:U_H\longrightarrow H^1_{\dR}(O_H),\qquad
 \chi_H(g)=[g^*\theta_H-\theta_H],
\]
is a surjective homomorphism.  Its graded support is the coefficient-degree
\(p-1\) component, and
\[
 H^1_{\dR}(O_H)
 =\bigoplus_{j=1}^{2r}k[x_j^{p-1}dx_j].
\]
\end{proposition}

\begin{proof}
For \(g,h\in U_H\),
\[
 (gh)^*\theta_H-\theta_H
 =h^*(g^*\theta_H-\theta_H)+(h^*\theta_H-\theta_H).
\]
The tangent-to-identity group acts trivially on the height-one de Rham
cohomology.  For a substitution congruent to the identity modulo
\(\fm^2\), each Cartier basis class \([x_j^{p-1}dx_j]\) is unchanged:
every new term is either exact or has a coefficient exponent outside the
height-one box.  Passing to cohomology gives additivity.

The displayed basis follows from the one-variable de Rham calculation.
For each
symplectic pair, the substitutions
\[
 p_i\longmapsto p_i+a q_i^{p-1},
 \qquad
 q_i\longmapsto q_i+a p_i^{p-1}
\]
are symplectic and realize its two basis directions.  Hence \(\chi_H\) is
surjective.  Since the basis forms have coefficient degree \(p-1\), the
graded support is concentrated there.
\end{proof}

\subsection{Homogeneous Poisson brackets}

Let \(O_{H,e}\) denote the degree-\(e\) homogeneous part of \(O_H\).

Generation of the positive part is the classical computation of
\cite{KostrikinShafarevich1969} and
\cite[Ch.~4, Theorems~7.2--7.5]{StradeFarnsteiner1988}.  In the height-one
truncation the exceptional degrees are the quadratic, flux, and socle
components.

\begin{proposition}[Homogeneous Poisson generation]
\label{prop:poisson-generation}
For \(4\le e\le D_H\),
\[
 \sum_{\substack{a+b=e+2\\a,b\ge3}}
 \{O_{H,a},O_{H,b}\}
 =
 \begin{cases}
 O_{H,e},&e<D_H,\\
 0,&e=D_H.
 \end{cases}
\]
Equivalently, the only nonzero homogeneous quotients in the abelianization
of the positive Hamiltonian filtration are the quadratic, flux, and socle
components:
\[
 \frac{\cH_d}
 {\sum_{a+b=d+1}[\cH_a,\cH_b]}
 \simeq
 \begin{cases}
 \cH_2,&d=2,\\
 H^1_{\dR}(O_H),&d=p-1,\\
 kX_M,&d=D_H-1,\\
 0,&\text{otherwise}.
 \end{cases}
\]
\end{proposition}

\begin{proof}
It suffices to treat monomials.  In one symplectic pair \((q,p)\), for a
local target \(q^Ap^B\) of degree \(d=A+B\), set
\[
 f_a=q^ap^{3-a},\qquad
 g_a=q^{A-a+1}p^{B+a-2}.
\]
Whenever the exponents lie in the height-one box,
\[
 \{f_a,g_a\}=\bigl(3(A+1)-a(d+2)\bigr)q^Ap^B.
\]
If \(A\le p-2\) and \(B\ge2\), take \(a=0\); if
\(B\le p-2\) and \(A\ge2\), take \(a=3\).  The resulting coefficients
are \(3(A+1)\) and \(-3(B+1)\), respectively, and are nonzero.  For
\(d\ge3\), simultaneous failure of these alternatives forces either
\(d\le2\) or \((A,B)=(p-1,p-1)\).  In local degree two the only remaining
case is \(\{qp^2,q\}=2qp\).

Now let
\[
 m=\prod_{i=1}^r q_i^{A_i}p_i^{B_i}\in O_{H,e},\qquad e\ge4,
\]
with \(m\ne M\).  If some nonsocle pair has local degree at least two,
use the preceding identity in that pair and multiply its second
Hamiltonian by the monomial in all remaining variables.  No other
symplectic pair contributes to the bracket, and the two Hamiltonians have
degrees \(3\) and \(e-1\).

Suppose instead that every nonsocle pair has degree at most one.  If one
of them contributes, say, \(q_i\), write \(m=q_iAB\), where \(A,B\)
involve the other pairs, \(\deg A\ge1\), \(\deg B\ge2\), and
\(\{A,B\}=0\).  Such a factorization is obtained by assigning singleton
pair factors wholly to one side and, when necessary, splitting a socle
factor \((q_jp_j)^{p-1}\) between two powers of \(q_jp_j\).  Then
\[
 \{q_i^2A,p_iB\}=-2m.
\]
The case of a single \(p_i\) is symmetric.  If every nonsocle pair is
empty, choose an empty pair \(j\).  At least one other pair is a socle
pair; factor \(m=AB\) with \(\deg A,\deg B\ge2\) and \(\{A,B\}=0\) by
splitting one socle factor.  Then
\[
 \{Aq_j,Bp_j\}=-m.
\]
Hence every nonsocle monomial of degree at least four is a Poisson bracket
of homogeneous elements of degrees at least three.

For the socle, put \(\int f=[M]f\).  Since
\(\int\partial_xh=0\) for every coordinate \(x\), integration by parts
gives
\[
 \int\partial_{q_i}f\,\partial_{p_i}g
 =-\int f\,\partial_{q_i}\partial_{p_i}g
 =\int\partial_{p_i}f\,\partial_{q_i}g.
\]
Hence \(\int\{f,g\}=0\).  Since \(O_{H,D_H}=kM\) and \(\int M=1\),
the bracket span in top degree is zero.

Finally, contraction with \(\omega_H\) identifies symplectic fields with
closed one-forms and Hamiltonian fields with exact one-forms.  Height-one
de Rham cohomology occurs only in coefficient degree \(p-1\); the bracket
calculation gives every exact component in degrees \(3,\ldots,D_H-2\).
Degree two receives no bracket of positive-degree fields, and the top
Hamiltonian line is \(kX_M\).
\end{proof}

\subsection{The congruence filtration and abelianization}

\begin{proposition}[Hamiltonian congruence filtration]
\label{prop:hamiltonian-realization}
For \(2\le d\le D_H-1\), the leading-field map induces an isomorphism
\[
 F^dU_H/F^{d+1}U_H\xrightarrow{\sim}\cH_d.
\]
Under these identifications, the initial field of a group commutator is
the bracket of the initial fields.  Both identifications commute with
extension of the perfect ground field, in particular with passage to every
\(\F_q\).
\end{proposition}

\begin{proof}
Linearizing \(g^*\omega_H=\omega_H\) gives a symplectic initial field,
and injectivity follows from the definition of the filtration.  For
surjectivity, first
consider \(\cE_2=X_{O_{H,3}}\).  Since \(3!\) is invertible, cubes of
linear forms span \(O_{H,3}=\Sym^3(V^*)\).  For a linear form \(\ell\),
set
\[
 Y_\ell=X_{\ell^3/3}=\ell^2X_\ell.
\]
Because \(X_\ell(\ell)=0\), one has
\(Y_\ell^m=\ell^{2m}X_\ell^m\); since \(\ell^p=0\), the exponential
\(\exp(aY_\ell)\) is a finite sum involving only \(m<p\).  The truncated
exponential identities are therefore valid in characteristic \(p\), and
\(\exp(aY_\ell)\) is symplectic.  On linear coordinates
\(Y_\ell^2=0\), so its action is
\[
 T_{\ell,a}^*(x)=x+aY_\ell(x).
\]
Products of these one-parameter subgroups realize \(\cE_2\).

Assume inductively that the exact components below coefficient degree
\(d\) have been realized, and let \(X_f\in\cE_d\),
\(3\le d\le D_H-2\).  By \cref{prop:poisson-generation},
\[
 f=\sum_\nu\{a_\nu,b_\nu\},
\]
with all corresponding Hamiltonian fields of smaller coefficient degree.
After lifting these fields, the product of their group commutators has
initial field \(\sum_\nu[X_{a_\nu},X_{b_\nu}]=X_f\).  Hence every exact
component below the socle is realized.

The top exact component is \(kX_M\).  Since \(X_M^2=0\) on linear
coordinates,
\[
 z_a^*(x)=x+aX_M(x)
\]
has inverse \(z_{-a}\); the identity \(\mathcal L_{X_M}\omega_H=0\),
together with the height-one degree bound on quadratic pullback terms,
shows that \(z_a\) is symplectic.  Finally,
\(\cH_d/\cE_d\) is nonzero only for \(d=p-1\), where its Cartier basis is
realized by
\[
 p_i\mapsto p_i+a q_i^{p-1},
 \qquad q_i\mapsto q_i+a p_i^{p-1}.
\]
Hence the leading-field map is surjective in every degree.

For substitutions with initial fields \(X\) and \(Y\), direct expansion
gives initial commutator \([X,Y]\) in coefficient degree
\(\deg X+\deg Y-1\).  The lifts and commutator calculation are defined over
the prime field and therefore commute with extension of the perfect ground
field.
\end{proof}

Let
\[
 j_{H,2}:U_H\longrightarrow Q=\Sym^3(V^*)
\]
be the quadratic Hamiltonian jet.

\begin{proposition}[Filtered abelianization]
\label{prop:hamiltonian-abelianization}
Let
\[
 D=[U_H,U_H]_{\alg},\qquad A_H=U_H/D.
\]
For the filtration induced on \(D\),
\[
 \gr_2D=0,
 \qquad
 \gr_dD=\cE_d\quad(3\le d\le D_H-2),
\]
under the leading-field identifications of
\cref{prop:hamiltonian-realization}.  It follows that the induced filtration
on \(A_H\) has only the following
possible nonzero successive quotients:
\[
 Q=\Sym^3(V^*),\qquad F=V^{(1)},\qquad T_M,
\]
of scalar weights \(1,p-2,N\), respectively, where \(T_M\) is either
zero or the socle line \(kX_M\).  The quotients \(Q\) and \(F\) are the
surjective images induced by \(j_{H,2}\) and \(\chi_H\).
\end{proposition}

\begin{proof}
The depth estimate
\[
 [F^aU_H,F^bU_H]\subseteq F^{a+b-1}U_H
\]
gives \(D\subseteq F^3U_H\), hence \(\gr_2D=0\).  Moreover
\(\chi_H(D)=0\).  Since \(\cH_d/\cE_d\) is supported only in degree
\(p-1\), every initial field of \(D\) is exact; thus
\(\gr_dD\subseteq\cE_d\).

Conversely, let \(3\le d\le D_H-2\) and \(X_f\in\cE_d\).  Write
\(f=\sum_\nu\{a_\nu,b_\nu\}\) by
\cref{prop:poisson-generation}, lift the smaller-degree Hamiltonian fields
by \cref{prop:hamiltonian-realization}, and take the product of the
corresponding group commutators.  It lies in \(D\) and has initial field
\[
 \sum_\nu[X_{a_\nu},X_{b_\nu}]=X_f.
\]
Hence \(\gr_dD=\cE_d\) throughout this range.

The filtration on \(A_H\) follows from
\(\cH_{p-1}/\cE_{p-1}\simeq V^{(1)}\) and
\(\cH_{D_H-1}=kX_M\).  The surjections \(j_{H,2}\) and \(\chi_H\) kill
\(D\) and induce the \(Q\)- and \(F\)-quotients.  Their scalar weights,
and that of the socle, are \(1\), \(p-2\), and \(D_H-2=N\).
\end{proof}

\subsection{Degree-tracked Poisson generation}

Throughout this subsection \(r\ge2\).  Recall
\(\cE_d=X_{O_{H,d+1}}\), so that \(\cE_d\) corresponds to \(O_{H,d+1}\) and
\(\cE_{D_H-1}=kX_M\).  In function degrees a bracket of \(O_{H,u}\) and
\(O_{H,v}\) lands in \(O_{H,u+v-2}\).

\begin{lemma}[Degree-tracked generation]
\label{lem:tracked-poisson}
Let \(j\ge2\) and \(j+2\le e\le D_H-1\).  Then
\[
 \sum_{\substack{u+v=e+2\\ u\ge3,\ v\ge j+1}}
 \{O_{H,u},O_{H,v}\}
 =O_{H,e}.
\]
\end{lemma}

\begin{proof}
Since \(v=e+2-u\le e-1\), the conditions \(u\ge3\) and \(v\ge j+1\) are
compatible precisely when \(e\ge j+2\).

It suffices to treat a monomial
\(m=\prod_{i=1}^rq_i^{A_i}p_i^{B_i}\) of degree \(e\), necessarily \(m\ne M\).
Call a pair \emph{socle} if \((A_i,B_i)=(p-1,p-1)\).

\emph{Generic case: some nonsocle pair has degree at least two.}
Rename that pair \((q,p)\) with exponents \((A,B)\), and let \(R=m/(q^Ap^B)\).
For \(\alpha+\beta=3\) one has the single-pair identity
\begin{equation}
\label{eq:tracked-plane}
 \{q^{\alpha}p^{\beta},\,q^{A-\alpha+1}p^{B-\beta+1}R\}
 =\bigl(\beta(A+1)-\alpha(B+1)\bigr)\,m ,
\end{equation}
valid whenever all displayed exponents lie in \([0,p-1]\); no other pair
contributes, since the first factor involves only \(q,p\).  The two factors
have degrees \(3\) and \(e-1\).  As \(e\ge j+2\) we have \(e-1\ge j+1\), so
this splitting is admissible for the stated index range.

If \(A\le p-2\) and \(B\ge2\), take \((\alpha,\beta)=(0,3)\); the coefficient
is \(3(A+1)\ne0\).  If \(B\le p-2\) and \(A\ge2\), take
\((\alpha,\beta)=(3,0)\); the coefficient is \(-3(B+1)\ne0\).  If
\(A,B\ge1\), the choices \((2,1)\) and \((1,2)\) give coefficients
\(A+1-2(B+1)\) and \(2(A+1)-(B+1)\), which vanish simultaneously only if
\(A\equiv B\equiv-1\), i.e.\ only for a socle pair.  Since the pair is
nonsocle of degree at least two, at least one alternative applies.

\emph{Boundary configurations.}  Two configurations are not reached by
\eqref{eq:tracked-plane}.

\begin{enumerate}[label=\textup{(B\arabic*)}]
\item Every nonsocle pair has degree at most one, and some nonsocle pair
contributes a single variable, say \(q_i\).  Write \(m=q_iAB\) with
\(\{A,B\}=0\), \(\deg A\ge1\), \(\deg B\ge2\), \(A,B\) in the remaining
pairs.  Then \(\{q_i^2A,\,p_iB\}=-2m\), with factor degrees
\(2+\deg A\) and \(1+\deg B\).
\item Every nonsocle pair is empty.  Choose an empty pair \(j_0\) and split a
socle factor as \(m=AB\) with \(\deg A,\deg B\ge2\), \(\{A,B\}=0\).  Then
\(\{Aq_{j_0},\,Bp_{j_0}\}=-m\), with factor degrees \(1+\deg A\) and
\(1+\deg B\).
\end{enumerate}

In the index range of the lemma the small factor must have degree at
most \(e+1-j\).  In (B1) this asks for
\(\deg A\le e-1-j\); one may take \(\deg A=1\), so that the small factor has
degree \(3\), whenever some pair other than \(i\) contributes a single
variable.  The only configurations not covered are therefore

\begin{enumerate}[label=\textup{(C\arabic*)}]
\item \(m=q_i\prod_{l\in S}(q_lp_l)^{p-1}\) with \(S\ne\emptyset\) and every
pair other than \(i\) a socle pair, together with the symmetric case in which
the singleton is \(p_i\);
\item \(m=\prod_{l\in S}(q_lp_l)^{p-1}\) with \(S\ne\emptyset\) and some pair
\(j_0\notin S\) empty.
\end{enumerate}

Both are settled by the following separation of variables.  If \(u,v\) involve
only one symplectic pair and \(X,Y\) only the remaining pairs, then
\begin{equation}
\label{eq:separation}
 \{uX,vY\}=\{u,v\}\,XY+uv\,\{X,Y\},
\end{equation}
because the \(k\)-th summand of the bracket contributes \(XY\{u,v\}_k\) for
\(k\) the distinguished pair and \(uv\{X,Y\}_k\) otherwise.  Moreover, for a
socle pair \(l\in S\) and any \(k\),
\[
 \{q_lp_l,(q_lp_l)^k\}=k(q_lp_l)^{k-1}\{q_lp_l,q_lp_l\}=0 ,
\]
so with \(X=q_lp_l\) and \(Y=m/(q_lp_lq_i^{\epsilon})\) one has
\(\{X,Y\}=0\), the remaining pairs being untouched.

In case (C2) fix \(l\in S\) and put
\[
 f=q_{j_0}\cdot q_lp_l,
 \qquad
 g=p_{j_0}\cdot m/(q_lp_l).
\]
Then \(\{X,Y\}=0\) and \(\{q_{j_0},p_{j_0}\}=-1\), so \eqref{eq:separation}
gives \(\{f,g\}=-m\), with \(\deg f=3\) and \(\deg g=e-1\).

In case (C1) fix \(l\in S\) and put
\[
 f=p_i\cdot q_lp_l,
 \qquad
 g=q_i^2\cdot m/(q_i\,q_lp_l).
\]
Here \(u=p_i\), \(v=q_i^2\) give \(\{u,v\}=2q_i\), so \eqref{eq:separation}
gives \(\{f,g\}=2m\), again with \(\deg f=3\) and \(\deg g=e-1\).  All
exponents remain in the height-one box: \(q_i^2\) requires \(2\le p-1\), and
the socle pair \(l\) appears in \(g\) with exponents \(p-2\).  Since
\(p\ge5\), the scalar \(2\) is invertible.  The case of a singleton \(p_i\)
is symmetric, with \(f=q_i\,q_lp_l\) and \(g=p_i^2\,m/(p_i\,q_lp_l)\).

In all cases the small factor has degree \(3\) and the large factor degree
\(e-1\ge j+1\), which is the required index range.
\end{proof}

\subsection{The Hamiltonian central series}

\begin{theorem}[Hamiltonian lower central series and class]
\label{thm:hamiltonian-series}
Let \(p\ge5\) and \(r\ge2\).  Let
\(Z_M\) be the socle line and \(K_H=\ker\chi_H\).  All groups below are the
reduced ones, so that the induced filtration is taken set-theoretically.  Then
for \(2\le j\le D_H-2\)
\[
 F^{j+1}U_H\cap K_H
 =\Gamma_j(U_H)\times Z_M ,
\]
and in particular
\[
 [U_H,U_H]_{\alg}\times Z_M=F^3U_H\cap K_H,
 \qquad
 \operatorname{cl}(U_H)=2r(p-1)-3 .
\]
Moreover, \(\Gamma_j(U_H)=1\) for \(j\ge D_H-2\).
\end{theorem}

\begin{proof}
By \cref{prop:hamiltonian-realization} the initial field of a group
commutator is the bracket of the initial fields, so it suffices to compute
in the graded Lie algebra \(\mathfrak g=\bigoplus_{d=2}^{D_H-1}\cH_d\).

Put \(\mathfrak b_j=\bigoplus_{d\ge j}\cH_d\).  By
\cref{lem:tracked-poisson} translated through \(\cE_d\leftrightarrow
O_{H,d+1}\), the degree-\(d\) component of \([\mathfrak g,\mathfrak b_j]\) is
\(\cE_d\) for \(j+1\le d\le D_H-2\).  It is zero for \(d\le j\) by degree and
zero for \(d=D_H-1\), since \(\int\{f,g\}=0\), no bracket reaches the top
function degree.  Hence
\[
 \gamma_j(\mathfrak g)=\bigoplus_{d=j+1}^{D_H-2}\cE_d
 \qquad(j\ge2).
\]
For \(j=2\) this is \(\bigoplus_{d=3}^{D_H-2}\cE_d\), which is exactly
\cref{prop:poisson-generation}.

The graded of \(F^{j+1}U_H\cap K_H\) is
\(\bigoplus_{d\ge j+1}\cE_d\), because \(\cH_d=\cE_d\) for \(d\ne p-1\) and
\(\ker\chi_H\) cuts \(\cH_{p-1}\) down to \(\cE_{p-1}\); the flux character
has graded support in degree \(p-1\) only, so for \(j+1>p-1\) the
intersection is all of \(F^{j+1}U_H\).  Comparing with the display above,
the two subgroups differ exactly in the top layer
\(\cE_{D_H-1}=kX_M=\Lie Z_M\).  Since \(\cH_{D_H}=0\), the line \(Z_M\) is
central in \(U_H\), and by \cref{thm:hamiltonian-main}(2) it meets
\(\Gamma_2(U_H)\supseteq\Gamma_j(U_H)\) trivially for \(r\ge2\).  The
product is therefore direct and exhausts \(F^{j+1}U_H\cap K_H\).

Finally \(\gamma_j(\mathfrak g)=0\) exactly when \(j+1>D_H-2\), so
\(\Gamma_j(U_H)=1\) exactly when \(j\ge D_H-2\), giving
\(\operatorname{cl}(U_H)=D_H-3\).
\end{proof}

\subsection{The derived series}

The class computed above permits derived length
\(\lceil\log_2(D_H-2)\rceil\), and \cref{thm:hamiltonian-derived} shows
that it is attained.  As for \cref{lem:balanced-special} in the special
case, the argument uses brackets of two factors of equal degree \(c\).

Let \(m=\prod q_i^{A_i}p_i^{B_i}\) have degree \(2c-2\), pick a
distinguished pair, and write \(m=q_i^Ap_i^BR\).  With
\(u=q_i^{\alpha}p_i^{\beta}\), \(v=q_i^{A-\alpha+1}p_i^{B-\beta+1}\),
\(s=\alpha+\beta\) and \(R=XY\), the separation identity
\eqref{eq:separation} gives
\[
 \{uX,vY\}=\bigl(\beta(A+1)-\alpha(B+1)\bigr)m+q_i^{A+1}p_i^{B+1}\{X,Y\},
\]
and \(\deg(uX)=\deg(vY)=c\) precisely when \(\deg X=c-s\).

The error term is \(uv\{X,Y\}\) with \(uv=q_i^{A+1}p_i^{B+1}\).  This term vanishes whenever the distinguished pair has exactly one exponent
equal to \(p-1\).

\begin{lemma}[Balanced brackets, saturated case]
\label{lem:balanced-saturated}
Let \(3\le c\) with \(2c-2\le D_H-1\), and let \(m\) be a monomial of degree
\(2c-2\) having a pair with exactly one exponent equal to \(p-1\).  Then
\(m\in\{O_{H,c},O_{H,c}\}\).
\end{lemma}

\begin{proof}
Rename that pair as pair \(1\) and say \(A=p-1\), \(B\le p-2\).  Then
\(uv=q_1^{p}p_1^{B+1}=0\), so \eqref{eq:separation} reduces to
\(\{uX,vY\}=\{u,v\}XY\) for \emph{every} factorization \(R=XY\); no
commuting condition is imposed, and \(\deg X\) may be any value in
\([0,\deg R]\).

The exponents of \(v\) are \(\gamma=A-\alpha+1=p-\alpha\) and
\(\delta=B-\beta+1\).  The condition \(\gamma\in[0,p-1]\) forces
\(\alpha\ge1\), and \(\delta\in[0,p-1]\) forces
\(\beta\le B+1\).  The coefficient is
\[
 \beta(A+1)-\alpha(B+1)\equiv-\alpha(B+1)\pmod p,
\]
nonzero for every \(\alpha\ge1\) since \(B\le p-2\).  Hence every
\(s=\alpha+\beta\in[1,p+B]\) is admissible.

It remains to find such an \(s\) with \(\deg X=c-s\in[0,\deg R]\), where
\(\deg R=2c-2-A-B=2c-p-1-B\).  The two conditions read
\[
 s\in\bigl[\max(1,\,p+B+1-c),\ \min(c,\,p+B)\bigr],
\]
and this interval is nonempty: \(p+B+1-c\le c\) is equivalent to
\(p+B+1\le2c\), which follows from
\(2c-2=\deg m\ge A+B=p-1+B\).  Choosing such an \(s\), splitting \(R\)
arbitrarily with \(\deg X=c-s\), and taking \(\alpha\ge1\) with
\(\alpha+\beta=s\) gives \(\{uX,vY\}=\lambda m\) with \(\lambda\ne0\).
The case \(B=p-1\), \(A\le p-2\) is symmetric.
\end{proof}

\begin{lemma}[Balanced brackets]
\label{lem:balanced-hamiltonian}
Let \(3\le c\) with \(2c-2\le D_H-1\).  Then
\[
 \{O_{H,c},O_{H,c}\}=O_{H,2c-2}.
\]
\end{lemma}

\begin{proof}
Let \(\mathfrak V\) be the span of the brackets.  Every monomial \(m\) of degree
\(2c-2\) satisfies \(m\ne M\), since \(2c-2<D_H\).  For such an \(m\) put
\[
 \Psi(m)=\sum_{l}(A_l+B_l)^2 ,
\]
a quantity bounded above on monomials of degree \(2c-2\).  We argue by
downward induction on \(\Psi\).

If some pair of \(m\) has exactly one exponent equal to \(p-1\), then
\cref{lem:balanced-saturated} gives \(m\in\mathfrak V\) with no error term.  So assume
every pair of \(m\) is either socle or has both exponents at most \(p-2\).
Let \(i\) be a nonsocle pair of maximal degree \(d_i=A+B\), which exists as
\(m\ne M\), and write \(m=q_i^Ap_i^BR\).  Let \(t\) be the total degree of the
nonsocle pairs of \(R\) and \(\sigma\) the number of its socle pairs, so that
\(\deg R=\sigma(2p-2)+t\) and \(2c-2=d_i+t+\sigma(2p-2)\).

\emph{Admissible \(s\).}  Since \(A,B\le p-2\) the exponent conditions on
\(\alpha,\beta\) reduce to \(0\le\alpha\le A+1\) and \(0\le\beta\le B+1\), so
\(s=\alpha+\beta\) may be any value in \([0,d_i+2]\).  With \(s\) fixed,
\(\lambda=s(A+1)-\alpha(d_i+2)\).  If \(d_i+2\not\equiv0\) this vanishes for
at most one residue \(\alpha\), while the admissible \(\alpha\) for a given
\(s\in[1,d_i+1]\) form a range of at least two integers; if
\(d_i+2\equiv0\) then \(d_i=p-2\), \(\lambda=s(A+1)\), and \(A+1\le p-1\) is
invertible while \(s\le p-1\).  In both cases every \(s\in[1,d_i+1]\) admits
a choice of \(\alpha\) with \(\lambda\ne0\).

\emph{Achievable \(\deg X\).}  Split each socle pair of \(R\) as
\((a,a)\) and \((p-1-a,p-1-a)\); such a split contributes any even value in
\([0,2p-2]\) and satisfies \(b_lA_l-a_lB_l=0\), so it creates no error.
Assign the nonsocle pairs of \(R\) wholly to \(X\) or to \(Y\), except that
one of them may be split arbitrarily.  If \(t\ge1\) the nonsocle part
contributes every value in \([0,t]\), and since \(\sigma(2p-2)\) is even the
two contributions together realize every \(\deg X\in[0,\deg R]\).

\emph{Choice of \(s\) when \(t\ge1\).}  We need
\(s\in[1,d_i+1]\) with \(c-s\in[0,\deg R]\), that is
\(s\in[\max(1,d_i+2-c),\min(d_i+1,c)]\).  This interval is nonempty: \(1\le
d_i+1\) and \(1\le c\), while \(d_i+2-c\le d_i+1\) and \(d_i+2-c\le c\)
because \(d_i\le2c-2\).

\emph{Choice of \(s\) when \(t=0\).}  Here \(R\) is a product of socle pairs,
so \(\deg X\) must be even.  If \(d_i\ge1\) then \(s\in\{1,2\}\) are both
admissible and have opposite parities, so one of \(\deg X=c-1,c-2\) is even;
both lie in \([0,\deg R]\) because \(2c-2=d_i+\sigma(2p-2)\) with
\(\sigma\ge1\) forces \(d_i\le c-2\), whence \(\deg R=2c-2-d_i\ge c\).  If
\(d_i=0\) then \(A=B=0\), the only admissible value is \(s=1\) with
\((\alpha,\beta)=(0,1)\) and \(\lambda=1\), and \(\deg X=c-1\).  In that case
\(2c-2=\sigma(2p-2)\); as \(p\) is odd, \(4\mid2p-2\), so \(4\mid2c-2\) and
\(c\) is odd.  Hence \(c-1\) is even and is an achievable value.

With these choices,
\[
 \{q_i^{\alpha}p_i^{\beta}X,\;q_i^{A-\alpha+1}p_i^{B-\beta+1}Y\}
 =\lambda m+q_i^{A+1}p_i^{B+1}\{X,Y\},
 \qquad\lambda\ne0,
\]
and both factors have degree \(c\).  The error term is supported on monomials
obtained from \(m\) by raising pair \(i\) by \((1,1)\) and lowering the single
shared nonsocle pair \(l\) by \((1,1)\); socle pairs contribute nothing to
\(\{X,Y\}\).  Such a monomial has
\[
 \Psi=\Psi(m)+8+4(d_i-d_l)\ \ge\ \Psi(m)+8 ,
\]
since \(l\) is nonsocle and \(i\) has maximal nonsocle degree.  If
\(\Psi(m)\) is maximal among monomials of degree \(2c-2\) the error term
therefore vanishes; otherwise it lies in \(\mathfrak V\) by the induction hypothesis.
In both cases \(\lambda m\in\mathfrak V\), hence \(m\in\mathfrak V\).
\end{proof}

\begin{theorem}[Hamiltonian derived series]
\label{thm:hamiltonian-derived}
Let \(p\ge5\) and \(r\ge2\).  Then \(U_H^{(j)}=\Gamma_{2^j}(U_H)\) for every
\(j\ge1\), and
\[
 \operatorname{dl}(U_H)
 =\bigl\lceil\log_2(D_H-2)\bigr\rceil
 =\bigl\lceil\log_2\bigl(2r(p-1)-2\bigr)\bigr\rceil .
\]
\end{theorem}

\begin{proof}
By \cref{thm:hamiltonian-series} the graded of \(\Gamma_j(U_H)\) is
\(\bigoplus_{d=j+1}^{D_H-2}\cE_d\), which under
\(\cE_d\leftrightarrow O_{H,d+1}\) is \(O_{H,[j+2,\,D_H-1]}\).  The inclusion
\(G^{(j)}\subseteq\gamma_{2^j}(G)\) holds in every group.  For the reverse,
induct on \(j\), and suppose that the graded of \(U_H^{(j)}\) is
\(O_{H,[c,\,D_H-1]}\) with \(c=2^j+2\).  Let \(\mathfrak B\) be the span of
the brackets \(\{O_{H,a},O_{H,b}\}\) with \(a,b\ge c\).  By
\cref{lem:balanced-hamiltonian}, \(O_{H,2c-2}\subseteq\mathfrak B\).  The
proof of \cref{lem:tracked-poisson} writes every monomial of degree \(e\),
\(4\le e\le D_H-1\), as a bracket whose first factor has degree \(3\), so
\(O_{H,e}=\{O_{H,3},O_{H,e-1}\}\).  By the Jacobi identity,
\[
 \{O_{H,3},\{O_{H,a},O_{H,b}\}\}
 \subseteq\{O_{H,a+1},O_{H,b}\}+\{O_{H,a},O_{H,b+1}\}
 \subseteq\mathfrak B ,
\]
and induction on \(e\) gives \(O_{H,e}\subseteq\mathfrak B\) for
\(2c-2\le e\le D_H-1\).  Lifting these brackets to commutators of elements
of \(U_H^{(j)}\) and removing leading terms through the finite congruence
filtration, as in the proof of \cref{prop:special-commutators}, shows that
the graded of \(U_H^{(j+1)}\) contains \(O_{H,[2^{j+1}+2,\,D_H-1]}\).  Hence
\(U_H^{(j+1)}=\Gamma_{2^{j+1}}(U_H)\).

Finally \(\Gamma_i(U_H)=1\) exactly when \(i\ge D_H-2\), so
\(U_H^{(j)}=1\) exactly when \(2^j\ge D_H-2\).
\end{proof}

\subsection{Finite-field points and the Frattini quotient}

\begin{lemma}[One-parameter generators]
\label{lem:hamiltonian-finite-generators}
Let \(p\ge5\), \(r\ge2\), and \(q=p^f\).  Choose linear forms
\(\ell_1,\ldots,\ell_t\) over \(\F_p\) such that
\(\ell_\nu^3/3\) form a basis of the cubic polynomials, where
\(t=\binom{2r+2}{3}\).  Then \(U_H(\F_q)\) is generated by the
\(\F_q\)-points of
\[
 T_{\ell_\nu,a},\qquad
 p_i\longmapsto p_i+a q_i^{p-1},\qquad
 q_i\longmapsto q_i+a p_i^{p-1},\qquad z_a
 \quad(a\in\F_q).
\]
Each displayed family is a homomorphism from \((\F_q,+)\).
\end{lemma}

\begin{proof}
Polarization of cubes, with \(2\) and \(3\) invertible, gives the chosen
basis over \(\F_p\).  In \(T_{\ell,a}\), the linear form \(\ell\) is
fixed, so substitution on the linear coordinates gives
\(T_{\ell,a}T_{\ell,b}=T_{\ell,a+b}\).  The same identity holds for
each flux shear and for \(z_a\).

Let \(B\) be the subgroup generated by these points.  The quadratic
initial fields of \(B\) span \(\cH_2\).  Inductively,
\cref{prop:poisson-generation} expresses every exact initial field below
the socle as a sum of brackets of already realized smaller-degree fields.
Products of their commutators in \(B\) realize that sum.  The flux
shears supply the nonexact directions in degree \(p-1\), and \(z_a\)
supplies the top line.  All coefficients and lifts are over \(\F_q\)
by \cref{prop:hamiltonian-realization}.  Thus \(B\) realizes every
congruence quotient.  Successively removing initial terms, up to
\(F^{D_H}U_H=1\), proves \(B=U_H(\F_q)\).
\end{proof}

\begin{proof}[Proof of \cref{thm:hamiltonian-finite}]
Put \(H=U_H(\F_q)\).  Abstract commutators give
\[
 \gamma_j(H)\subseteq\Gamma_j(U_H)(\F_q).
\]
We prove the reverse
inclusion by induction on \(j\).  For \(j=2\), the monomial bracket
identities of \cref{prop:poisson-generation}, lifted over \(\F_q\),
realize every \(\cE_d\) with \(3\le d\le D_H-2\) inside
\([H,H]\).  For the induction step, the cubic bracket expressions in
\cref{lem:tracked-poisson} use a quadratic initial field and an exact
field already realized in \(\gamma_{j-1}(H)\).  They realize every
\(\cE_d\) with \(j+1\le d\le D_H-2\) in \(\gamma_j(H)\).
Scaling one lifted field supplies each required coefficient in \(\F_q\).

Given an \(\F_q\)-point of \(\Gamma_j(U_H)\), cancel its initial field by these
finite-point commutators, successively through the congruence filtration.
The final remainder belongs to
\(Z_M(\F_q)\cap\Gamma_2(U_H)(\F_q)=1\), by
\cref{thm:hamiltonian-main}.  Hence
\(\gamma_j(H)=\Gamma_j(U_H)(\F_q)\).  In particular,
\[
 \gamma_j(H)\times Z_M(\F_q)
 =(F^{j+1}U_H\cap K_H)(\F_q)
 \quad(2\le j\le D_H-2).
\]
For the derived series, the cases \(a=0,1\) are already established.
Suppose \(H^{(a)}=\gamma_{2^a}(H)\), with \(a\ge1\).  The balanced
brackets and Jacobi filling in \cref{thm:hamiltonian-derived} are defined
over \(\F_p\); lifted to \(H^{(a)}\), they realize every exact initial
field in degrees \(2^{a+1}+1,\ldots,D_H-2\).  The same successive
cancellation, with trivial final socle remainder, gives
\(\gamma_{2^{a+1}}(H)\subseteq H^{(a+1)}\).  The reverse inclusion
holds in every group, completing the induction.

By \cref{lem:hamiltonian-finite-generators}, the abelianization of \(H\)
is generated by elements of order dividing \(p\).  Thus
\(H^p\subseteq[H,H]\) and \(\Phi(H)=[H,H]\).
For every \(j\), the quotient
\(\gamma_j(H)/\gamma_{j+1}(H)\) is generated by the images of
left-normed commutators of length \(j\).  Modulo \(\gamma_{j+1}(H)\)
this commutator is a homomorphism in each variable and factors through
\(H/[H,H]\).  Its image is therefore killed by \(p\), so
\(\gamma_j(H)^p\subseteq\gamma_{j+1}(H)\).  Induction in the defining
recurrence for \(P_j(H)\) now gives \(P_j(H)=\gamma_j(H)\).

The congruence quotients of \(H\) are \(\cH_d(\F_q)\).
Those of \([H,H]\) are \(\cE_d(\F_q)\) for
\(3\le d\le D_H-2\), with no top socle term.  Consequently
\[
 |H/[H,H]|=q^{\dim\cH_2+2r+1}
 =q^{\binom{2r+2}{3}+2r+1}.
\]
The \(t\) quadratic, \(2r\) flux, and one socle families in
\cref{lem:hamiltonian-finite-generators} define a surjective homomorphism
from \((\F_q,+)^{t+2r+1}\) to this abelianization.  The order calculation
makes it an isomorphism, and gives the stated generator number.  Finally,
the last nonzero lower-central term has degree \(D_H-2\), and
\(\Gamma_j(U_H)=1\) exactly for \(j\ge D_H-2\).  The finite-point
equalities prove the nilpotency class and derived length formulas.
\end{proof}

\begin{corollary}[Power depth at the socle]
\label{cor:hamiltonian-socle-power}
Let \(p\ge5\), \(r\ge2\), and \(q=p^f\).  If \(a\ge2\) and
\(p(a-1)\ge D_H-2\), then every element of
\(F^aU_H(\F_q)\) has order dividing \(p\).
\end{corollary}

\begin{proof}
For \(g\in F^aU_H(\F_q)\), write \(g^*=1+\Delta\).  The operator
\(\Delta\) raises degree by at least \(a-1\), and
\((g^*)^p=1+\Delta^p\).  Hence
\[
 g^p\in F^{1+p(a-1)}U_H(\F_q)
 \subseteq F^{D_H-1}U_H(\F_q)=Z_M(\F_q).
\]
By \cref{thm:hamiltonian-finite}, \(g^p\in[U_H(\F_q),U_H(\F_q)]\).
The intersection with \(Z_M(\F_q)\) is trivial, so \(g^p=1\).
\end{proof}

\subsection{The Calabi homomorphism}

Recall \(K_H=\ker\chi_H\).  For \(g\in K_H\), choose
\(S_g\in O_H/k\) satisfying
\[
 dS_g=g^*\theta_H-\theta_H.
\]

\begin{lemma}[Vanishing of the primitive trace]
\label{lem:calabi-primitive-trace}
Let \(h\in K_H\) and let \(\alpha\) be a closed one-form.  If
\[
 dR=(h^*-1)\alpha,
\]
then
\[
 \int R=0.
\]
\end{lemma}

\begin{proof}
For a closed \(\alpha\) and \(g\in U_H\), the form
\((g^*-1)\alpha\) is exact because \(U_H\) acts trivially on
\(H^1_{\dR}(O_H)\).  Let \(R_{g,\alpha}\in O_H/k\) be its unique
primitive and put \(\tau_\alpha(g)=\int R_{g,\alpha}\).  Since
\(d:O_H/k\to B^1_{\dR}(O_H)\) is a linear isomorphism,
\(\tau_\alpha\) is regular.  Moreover
\[
 R_{g_1g_2,\alpha}=g_2^*R_{g_1,\alpha}+R_{g_2,\alpha}\pmod k,
\]
and symplectic change of variables preserves \(\int\).  Hence
\(\tau_\alpha:U_H\to\Ga\) is a homomorphism and vanishes on
\(D=[U_H,U_H]_{\alg}\).

If \(\alpha=dA\), then \(R_{g,\alpha}=g^*A-A\), so
\(\tau_\alpha=0\).  It therefore suffices to consider the Cartier basis forms
\(\alpha_j=x_j^{p-1}dx_j\).  Set
\[
 \mathscr P=K_H\cap\bigcap_j\ker\tau_{\alpha_j};
\]
this is a closed subgroup containing \(D\).

For the quadratic lifts
\(T_{\ell,a}=\exp(aY_\ell)\) of
\cref{prop:hamiltonian-realization}, Cartan's formula gives
\[
 T_{\ell,a}^*\alpha_j-\alpha_j
 =d\left(
   \sum_{1\le m<p}\frac{a^m}{m!}
   \mathcal L_{Y_\ell}^{m-1}\bigl(\alpha_j(Y_\ell)\bigr)
 \right).
\]
The first term has zero socle coefficient: for any Hamiltonian field
\(X_f\), the function \(x_j^{p-1}X_f(x_j)\) could contain the socle only
if \(f\) had exponent \(p\) in the coordinate paired with \(x_j\).
For the higher terms, integration by parts gives
\[
 \int X(B)=-\int(\operatorname{div}X)B=0
\]
for every Hamiltonian field \(X\).  Hence the quadratic lifts lie in \(\mathscr P\).  The same calculation applies to
\(z_a=\exp(aX_M)=1+aX_M\), since \(X_M^2=0\) by degree.

By \cref{prop:hamiltonian-abelianization},
\(\gr_dD=\cE_d\) for \(3\le d\le D_H-2\).  The quadratic lifts span
\(\cE_2=\cH_2\), and \(z_a\) spans \(\cE_{D_H-1}=kX_M\).  Hence
\(\gr_d\mathscr P\supseteq\cE_d\) in every degree.  On the other hand,
\(\gr_dK_H=\cE_d\): in degree \(p-1\) this is the kernel of the Cartier
projection, and elsewhere \(\cH_d=\cE_d\).  Thus
\(\gr_d\mathscr P=\gr_dK_H\) for all \(d\), and the finite congruence
filtration gives \(\mathscr P=K_H\).
\end{proof}

\begin{proposition}[Calabi homomorphism]
\label{prop:calabi}
The formula
\[
 \Cal_H(g)=\int S_g
\]
defines a homomorphism \(K_H\to\Ga\), invariant under conjugation by
\(U_H\).  On the socle subgroup \(Z_M=\{z_a\}\),
\[
 \Cal_H(z_a)=a(r+1).
\]
\end{proposition}

\begin{proof}
The primitive \(S_g\) is unique modulo constants.  Since every symplectic
substitution has Jacobian one,
\[
 \int h^*f=\int f.
\]
The identity
\[
 S_{gh}=h^*S_g+S_h\pmod k
\]
therefore makes \(\Cal_H\) additive.

Set
\[
 c(u)=u^*\theta_H-\theta_H.
\]
For \(k\in U_H\), the form \(c(k)\) is closed.  If \(g\in K_H\),
\cref{lem:calabi-primitive-trace} gives \(R\) with
\[
 dR=(g^*-1)c(k),
 \qquad
 \int R=0.
\]
Using
\[
 c(ab)=b^*c(a)+c(b),
 \qquad
 c(k^{-1})=-(k^{-1})^*c(k),
\]
one obtains
\[
 c(kgk^{-1})
 =(k^{-1})^*\!\left(c(g)+(g^*-1)c(k)\right)
 =d\Bigl((k^{-1})^*(S_g+R)\Bigr).
\]
Thus
\[
 \Cal_H(kgk^{-1})
 =\int(k^{-1})^*(S_g+R)
 =\Cal_H(g),
\]
so \(\Cal_H\) is conjugation-invariant.

For \(z_a\), every nonlinear term of the pullback vanishes in the
height-one box.
Cartan's formula for the symmetric primitive \(\theta_H\) gives
\[
 z_a^*\theta_H-\theta_H
 =d\!\left(\left(1-\frac{\deg M}{2}\right)aM\right)
 =d\bigl(a(r+1)M\bigr),
\]
because \(\deg M=2r(p-1)\).  Hence
\(\Cal_H(z_a)=a(r+1)\).
\end{proof}

\begin{lemma}[Calabi vanishing on commutators]
\label{lem:calabi-derived}
If \(r\ge2\), then
\[
 \Cal_H\bigl([U_H,U_H]_{\alg}\bigr)=0.
\]
\end{lemma}

\begin{proof}
Since \(\chi_H\) is a homomorphism, the derived subgroup is contained in
\(K_H\).  Conjugation invariance gives, for \(u\in U_H\) and
\(k\in K_H\),
\[
 \Cal_H([u,k])
 =\Cal_H(uku^{-1})-\Cal_H(k)=0.
\]
The quotient \(U_H/K_H\simeq H^1_{\dR}(O_H)\) is generated by the flux
shears of \cref{prop:hamiltonian-flux}.  Modulo
\([U_H,K_H]_{\alg}\), the derived subgroup is therefore generated by
pairwise commutators of these lifts.  Shears supported on distinct
symplectic pairs commute.  A commutator of the two shear types supported
on the same pair fixes all other coordinates, so its normalized Calabi
primitive may be chosen in the subalgebra \(k[q_i,p_i]\).  When
\(r\ge2\), such a primitive has zero coefficient at
\[
 M=\prod_jq_j^{p-1}p_j^{p-1}.
\]
Hence \(\Cal_H\) vanishes on these commutators.
\end{proof}

\begin{corollary}
\label{cor:hamiltonian-nonresonant}
If \(r\ge2\) and \(p\nmid(r+1)\), then
\[
 \bigl(Z_M\cap[U_H,U_H]_{\alg}\bigr)_{\red}=1.
\]
If \(r=1\), then \(Z_M\subseteq[U_H,U_H]_{\alg}\).
\end{corollary}

\begin{proof}
For \(r\ge2\), by \cref{prop:calabi,lem:calabi-derived},
\(\Cal_H\) is nonzero on \(Z_M\) and zero on the derived subgroup.  If
\(r=1\), the two flux shears on the unique symplectic pair have a
commutator with leading Hamiltonian
\(q_1^{p-1}p_1^{p-1}=M\); there is no higher correction degree, so they
generate the socle line.
\end{proof}

\section{The resonant Hamiltonian case}
Assume throughout this section that \(r\ge2\) and
\[
 r+1=hp,\qquad h\ge1.
\]
Thus
\[
 N=D_H-2=2p\bigl(h(p-1)-1\bigr).
\]

\subsection{Linearization and the socle line}

Let
\[
 \Gamma_1=U_H,\qquad
 \Gamma_{j+1}=[U_H,\Gamma_j]_{\alg},\qquad
 C_j=\Gamma_j/\Gamma_{j+1}.
\]

\begin{proposition}[Linear lower-central quotients]
\label{prop:lower-central-vector-groups}
Multiplication by \(p\) on
\[
 A_H=U_H/[U_H,U_H]_{\alg}
\]
is zero.  For every \(j\ge1\), the group \(C_j\) is smooth, connected,
commutative, unipotent, and killed by \(p\).  If
\[
 T=T_0\times T_{\mathrm{sp}}
\]
is the product of scalar dilation and the standard diagonal symplectic
torus, then \(C_j\) admits a \(T\)-equivariant vector-group structure.
The left-normed \(j\)-fold commutator induces a \(T\)-equivariant
multi-homomorphism
\[
 w_j:A_H^j\longrightarrow C_j
\]
whose image generates \(C_j\) as a closed subgroup.
\end{proposition}

\begin{proof}
Let \(A_H^{\ge d}\) be the image of \(F^dU_H\) in \(A_H\).  The
multiplication morphism \([p]:A_H\to A_H\) is \(T_0\)-equivariant,
preserves this finite filtration, and has zero differential.  Suppose that
\([p]\ne0\), and choose \(d\) maximal with
\([p](A_H^{\ge d})\ne1\).  Then \([p]\) vanishes on
\(A_H^{\ge d+1}\).  Choose \(m\) maximal such that
\([p](A_H^{\ge d})\subseteq A_H^{\ge m}\).  Projection to the next
quotient gives a nonzero \(T_0\)-equivariant homomorphism
\[
 A_H^{\ge d}/A_H^{\ge d+1}
 \longrightarrow
 A_H^{\ge m}/A_H^{\ge m+1}.
\]
By \cref{prop:hamiltonian-abelianization}, the nonzero source and target
have scalar weights \(a,b\in\{1,p-2,N\}\).  Since the differential of
\([p]\) is zero, \cref{lem:frobenius-decomposition} gives a nonzero
Frobenius component only if
\[
 b=p^ea,\qquad e\ge1.
\]
For \(a\in\{1,p-2\}\), the integer \(p^ea\) is odd and at least \(p\):
it is neither \(1\) nor \(p-2\), and it cannot equal the even integer
\(N\).  For \(a=N\), one has \(p^eN>N\), again impossible.  Hence \([p]=0\) on
\(A_H\).

The image of the left-normed \(j\)-fold commutator word on \(U_H^j\) is
connected and contains the identity; its closed subgroup closure is
\(\Gamma_j\).  Hence every \(\Gamma_j\), and therefore every \(C_j\), is
connected.  By definition of \([G,H]_{\alg}\), the \(\Gamma_j\) are reduced; over the
perfect field \(k\) they and their quotients are smooth.  The quotient
\(C_j\) is commutative and unipotent.

Modulo \(\Gamma_{j+1}\), the commutator word is a homomorphism in each
variable and depends only on the classes in \(A_H\).  It thus induces the
\(T\)-equivariant multi-homomorphism \(w_j\), whose image generates
\(C_j\).  Since \([p]A_H=0\), the image of \(w_j\) lies in the
closed subgroup \(C_j[p]=\ker([p]:C_j\to C_j)\).  Hence
\[
 C_j=C_j[p].
\]

All scalar weights in the congruence filtration of \(U_H\) are positive,
and the same holds on every subquotient \(\Lie(C_j)\).  Hence no
\(T\)-character occurring in
\(\Lie(C_j)\) is trivial.  The torus-linearization theorem for smooth
connected commutative \(p\)-torsion unipotent groups
\cite[Theorem B.4.3]{CGP2015} therefore supplies a \(T\)-equivariant
vector-group structure on \(C_j\).
\end{proof}

\begin{lemma}[Retraction onto the socle line]
\label{lem:socle-retraction}
Suppose that \(Z_M\cap\Gamma_j\) contains a nontrivial geometric point,
and choose \(j\) maximal with this property.  Then
\(Z_M\subseteq\Gamma_j\), and the quotient map induces a
\(T\)-equivariant linear closed immersion
\[
 \iota_j:Z_M\hookrightarrow C_j.
\]
It admits a \(T\)-equivariant linear retraction
\[
 \rho_j:C_j\longrightarrow Z_M.
\]
Thus
\[
 \Phi_j=\rho_j\circ w_j:A_H^j\longrightarrow Z_M
\]
is a nonzero multi-homomorphism.
\end{lemma}

\begin{proof}
We may assume that \(k\) is algebraically closed.  The closed intersection
\(Z_M\cap\Gamma_j\) is stable under scalar dilation, and
\[
 t\cdot z_a=z_{t^Na}.
\]
For \(a\ne0\), the map \(t\mapsto t^N\) is surjective on \(\Gm\), so
the orbit of \(z_a\) is the punctured socle line and its closure is all of
\(Z_M\).  Hence \(Z_M\subseteq\Gamma_j\) after base change; since both
subgroups are reduced, this is a scheme-theoretic inclusion, and faithful
flat descent gives the inclusion over \(k\).

If the induced map \(\iota_j:Z_M\to C_j\) were zero, then
\(Z_M\subseteq\Gamma_{j+1}\), contrary to the maximality of \(j\).
Use the vector-group coordinates of
\cref{prop:lower-central-vector-groups}.  As a homomorphism from \(\Ga\),
\(\iota_j\) has the form
\[
 \iota_j(a)=\sum_{e\ge0}v_ea^{p^e}.
\]
Scalar equivariance places \(v_e\) in weight \(Np^e\).  Every scalar weight of \(\Lie(C_j)\) is inherited from the congruence
filtration of \(U_H\), whose weights lie between \(1\) and \(N\).  Hence
\(v_e=0\) for
\(e>0\), while \(v_0\ne0\).  Hence
\[
 \iota_j(a)=av_0,
\]
so \(\iota_j\) is a linear closed immersion with trivial scheme kernel.

The scalar-weight-\(N\) subspace of \(\Lie(U_H)\) is the socle line
\(kX_M\); therefore the corresponding weight space of \(C_j\) has
dimension at most one.  Since it contains \(v_0\), it is exactly
\(kv_0\).  The projection of the \(T\)-module \(C_j\) onto this scalar
weight space is \(T\)-equivariant, because \(T_{\mathrm{sp}}\) commutes
with \(T_0\).  Composing that projection with \(\iota_j^{-1}\) gives
\(\rho_j\).

If \(\rho_j\circ w_j\) were zero, the image of \(w_j\) would generate a
closed subgroup contained in \(\ker\rho_j\).  Since that image generates
\(C_j\), this would force \(C_j\subseteq\ker\rho_j\), contrary to
\(\rho_j\iota_j=\operatorname{id}_{Z_M}\).  Hence \(\Phi_j\ne0\).
\end{proof}

\subsection{Frobenius weights and higher commutators}

By \cref{prop:lower-central-vector-groups}, \(A_H=C_1\) is a
\(T\)-equivariant vector group.  Its \(T\)-stable filtration from
\cref{prop:hamiltonian-abelianization} has successive quotients
\[
 Q=\Sym^3(V^*),\qquad F=V^{(1)},\qquad T_M,
\]
of scalar weights \(1,p-2,N\).  Let \(q_H:U_H\to A_H\) be the quotient map.
The \(T\)-equivariant homomorphism \(q_H|_{Z_M}:\Ga\to A_H\) has additive
polynomial form \(q_H(z_a)=\sum_{e\ge0}v_ea^{p^e}\).  Scalar equivariance
places \(v_e\) in weight \(Np^e\); since all weights of \(A_H\) are at most
\(N\), only \(v_0\) can occur.  Thus the image of \(Z_M\), which is the top
filtration step, is either zero or a linear \(T\)-submodule of weight \(N\);
denote it by \(T_M\).  Since a torus is linearly reductive, fix a \(T\)-equivariant splitting
\begin{equation}
\label{eq:AH-weight-splitting}
 A_H\simeq Q\oplus F\oplus T_M.
\end{equation}
For \(j\ge2\), both \(w_j\) and \(\Phi_j\) vanish whenever one argument
lies in \(T_M\), because \(Z_M\) is central.  Every nonzero Frobenius-homogeneous component
of \(\Phi_j\) therefore factors through a product of copies of \(Q\) and
\(F\).

Applying \cref{lem:frobenius-decomposition}, a component with source
weights \(a_\nu\in\{1,p-2\}\) and Frobenius exponents \(e_\nu\ge0\)
can be nonzero only if its scalar weight equals that of the socle target:
\begin{equation}
\label{eq:frobenius-weight}
 \sum_{\nu=1}^j a_\nu p^{e_\nu}=N.
\end{equation}

Let \(\mathfrak s\) be the symplectic
fields, \(\mathfrak h=X_{O_H/k}\), and
\(\mathfrak h_0=X_{\ker\int}\).  For \(\Delta\subseteq U_H\), let
\(I_\Delta\subseteq k[\Delta]\) be the augmentation ideal acting through \(\Ad\).

\begin{lemma}[Augmentation depth]
\label{lem:two-augmentation}
For every \(\Delta\subseteq U_H\),
\[
 I_\Delta\mathfrak s\subseteq\mathfrak h,
 \qquad
 I_\Delta\mathfrak h\subseteq\mathfrak h_0,
 \qquad
 I_\Delta\mathfrak h_0\subseteq\mathfrak h_0.
\]
In particular, \(I_\Delta^m\mathfrak s\subseteq\mathfrak h_0\) for \(m\ge2\).
\end{lemma}

\begin{proof}
The action on
\(\mathfrak s/\mathfrak h\simeq H^1_{\dR}(O_H)\) is trivial, giving the
first inclusion.  On \(\mathfrak h\), the functional
\(\lambda(X_f)=\int f\) is \(\Delta\)-invariant: if
\(\Ad(g)X_f=X_{\rho_gf}\), then \(\int\rho_gf=\int f\) by symplectic
change of variables.  Hence
\(I_\Delta\mathfrak h\subseteq\ker\lambda=\mathfrak h_0\), and
\(\mathfrak h_0\) is \(\Delta\)-stable.
\end{proof}

\begin{lemma}[Fox depth]
\label{lem:fox-depth}
Let
\[
 W_j=[\cdots[[x_1,x_2],x_3],\ldots,x_j]
\]
be the left-normed commutator in the free group \(\mathfrak F\), with
\([a,b]=aba^{-1}b^{-1}\), and let
\(I_{\mathfrak F}\subseteq\mathbf Z[\mathfrak F]\) be the augmentation ideal.  Then
\[
 W_j-1\in I_{\mathfrak F}^{\,j},
 \qquad
 \frac{\partial W_j}{\partial x_i}\in I_{\mathfrak F}^{\,j-1}
 \quad(1\le i\le j).
\]
\end{lemma}

\begin{proof}
Induct on \(j\).  Write \(a=W_{j-1}\), \(b=x_j\), and
\(W_j=[a,b]\).  The identity
\[
 W_j-1=(ab-ba)a^{-1}b^{-1}
\]
and
\[
 ab-ba=(a-1)(b-1)-(b-1)(a-1)
\]
give \(W_j-1\in I_{\mathfrak F}^j\) by induction.  Fox differentiation
also gives
\[
 \frac{\partial W_j}{\partial x_i}
 =(1-aba^{-1})\frac{\partial a}{\partial x_i}
 +(a-W_j)\frac{\partial b}{\partial x_i}.
\]
If \(i<j\), the second term vanishes, while the first lies in
\(I_{\mathfrak F} I_{\mathfrak F}^{j-2}=I_{\mathfrak F}^{j-1}\).  If \(i=j\), the first term vanishes and
\(a-W_j\in I_{\mathfrak F}^{j-1}\).  The two assertions follow.
\end{proof}

Put \(U=U_H\), \(D=D_H\), \(G=\operatorname{Sp}_{2r}\), and
\(O_d=O_{H,d}\).

\begin{lemma}[The single-weight prefix]
\label{lem:single-weight-prefix}
Suppose that $j\ge3$ is maximal with $Z_M\subseteq\Gamma_j$.
Then $C_{j-1}$ is a $G$-equivariant vector group with
\[
 C_{j-1}\simeq O_{j+1},
 \qquad \text{scalar weight }j-1.
\]
If $j<N$, the two scalar-weight spaces of $C_j$ are
$O_{j+2}$, of weight $j$, and $Z_M$, of weight $N$.
If $j=N$, then $C_j=Z_M$.
\end{lemma}

\begin{proof}
For smooth reduced subgroups one has
\[
 [\Lie U,\Lie\Gamma_{m-1}]\subseteq\Lie\Gamma_m.
\]
Indeed, differentiate $[u,g]$ in $u$ at the identity and then
differentiate in $g$. The commutator morphism factors through the
reduced commutator subgroup because its source is reduced.
The cubic bracket identities in \cref{lem:tracked-poisson}
therefore give, by induction,
\[
 \bigoplus_{d=m+1}^{D-2}\cE_d
 \ \subseteq\ \Lie\Gamma_m
 \ \subseteq\
 \bigoplus_{d=m+1}^{D-2}\cE_d\oplus kX_M
 \qquad(m\ge2).
\]
The second inclusion follows from
$\Gamma_m\subseteq F^{m+1}U\cap K_H$.
Both $\Gamma_{j-1}$ and $\Gamma_j$ contain $Z_M$, so the top
line cancels in their Lie quotient. Smoothness of the subgroup
and quotient gives the exact Lie sequence, whence
\[
 \Lie C_{j-1}=\cE_j=X_{O_{j+1}},
\]
of scalar weight $j-1$. The map from potentials to fields is injective
in this positive degree and is $G$-equivariant.

In the chosen vector-group coordinates every additive polynomial
term in the $G$-action commutes with scalar dilation. A nonlinear
term would require $(j-1)p^b=j-1$ for $b>0$.
Thus the $G$-action on $C_{j-1}$ is linear and equals its Lie action.

The linear closed immersion $Z_M\hookrightarrow C_j$ supplied by
\cref{lem:socle-retraction} has injective differential. Consequently
$\Lie\Gamma_{j+1}$ does not contain $kX_M$. The Lie calculation
therefore gives the stated weight spaces of $C_j$.
\end{proof}

\begin{lemma}[A symplectic cocycle calculation]
\label{lem:symplectic-cocycle-vanishing}
For $0\le d\le D$ with $2(p-1)\mid d$, and every $a\ge0$,
\[
 H^1(G,O_d^{(a)})=0.
\]
Here $H^1$ denotes rational group cohomology.
\end{lemma}

\begin{proof}
First let $a=0$ and let $c:G\to O_d$ be a regular cocycle.
After subtracting a coboundary, its restriction to the diagonal
torus is zero. Its differential, still denoted $c$, is then a
torus-equivariant Lie cocycle, zero on the Cartan subalgebra.
Put
\[
 E_i=q_i\partial_{p_i},\quad F_i=p_i\partial_{q_i},\quad
 H_i=[E_i,F_i].
\]
The weight $2\varepsilon_i$ part of $O_d$ has monomials
$q_i^{b+2}p_i^b\prod_{\ell\ne i}(q_\ell p_\ell)^{a_\ell}$,
where $0\le b\le p-3$; the corresponding weight
$-2\varepsilon_i$ monomials interchange $q_i,p_i$.
The identity
\[
 E_i c(F_i)-F_i c(E_i)=c(H_i)=0
\]
equates their coefficients: the common factor $b+2$ is nonzero.
Both coefficients are therefore obtained from a balanced potential
by division by $b+1$.

These potentials can be chosen simultaneously. For a fixed balanced
monomial $\prod(q_i p_i)^{a_i}$, the coefficient prescribed by the
$i$-th equation is divided by $a_i$, for
$1\le a_i\le p-2$. If two indices $i,\ell$ are interior, the relation
$[E_i,E_\ell]=0$ equates these divided coefficients, because the
coefficient of their common image is multiplied by $a_i a_\ell\ne0$.
If there is no interior index, prescribe coefficient zero.
Thus a single balanced $f\in O_d$ makes
$c(E_i)=E_i f$ and $c(F_i)=F_i f$ for every $i$.
Subtract this coboundary.

The group cocycle now vanishes on all long-root groups as well.
On a long-root group its torus-equivariant coefficient of $t^n$
has weight $2n\varepsilon_i$, so
$1\le n\le(p-1)/2<p$. The cocycle identity gives
$c'(t)=u_i(t)c'(0)=0$, and a polynomial of degree less than $p$
with zero derivative is constant. Its value at zero is zero.
Hence $c$ vanishes on $L=(\mathrm{SL}_2)^r$, and its remaining
Lie cocycle is $L$-equivariant.

For one pair, the invariant vectors in its truncated algebra occur
only in degrees $0$ and $2p-2$. The natural two-dimensional module
maps into that algebra only in degrees $1$ and $2p-3$.
For the first assertion, a torus-weight-zero monomial is
$(qp)^b$; being killed by $q\partial_p$ forces $b=0$ or $p-1$.
For the second, the highest-weight-one monomials are
$q^{b+1}p^b$; being killed by $q\partial_p$ forces $b=0$ or $p-2$.
The natural module is generated by its highest vector.

The off-diagonal block of $\mathfrak{sp}_{2r}$ for pairs $i,\ell$
is the external tensor product $V_i\otimes V_\ell$.
The tensor-Hom factorization for a product of groups, together with
the preceding pairwise calculation, classifies its $L$-equivariant
maps to $O_d$.
Since $2(p-1)\mid d$, the two distinguished pairs must have one
degree $2p-3$ and one degree $1$. The alternatives with both degrees
low or both high have residues $2$ and $-2$ modulo $2(p-1)$ and
are excluded. All other pairs are empty or saturated.

Write $d=2s(p-1)$ and $m_i=(q_i p_i)^{p-1}$.
For $|S|=s$, $i\in S$, and $\ell\notin S$, let
$a_{S;i,\ell}$ be the coefficient in $c(E_{i\ell})$ of
\[
 M_{S;i,\ell}
 =q_i^{p-1}p_i^{p-2}p_\ell
   \prod_{b\in S\setminus\{i\}}m_b,
 \qquad
 E_{i\ell}=q_i\partial_{q_\ell}-p_\ell\partial_{p_i}.
\]
These coefficients determine the remaining cochain by
$L$-equivariance.

If $i\in S$ and $\ell,b\notin S$ are distinct, use
$[E_{i\ell},E_{\ell b}]=E_{ib}$ and compare
$M_{S;i,b}$. The term $E_{i\ell}c(E_{\ell b})$ contributes nothing:
an empty $i$-pair cannot become a high pair in one differentiation,
whereas a saturated $i$-pair either is killed by multiplication by
$q_i$ or leaves a nonempty $\ell$-pair. The other term contributes
only $E_{\ell b}M_{S;i,\ell}=-M_{S;i,b}$.
Thus $a_{S;i,b}=a_{S;i,\ell}$.

If $i,i'\in S$ are distinct and $\ell\notin S$, use
$[E_{ii'},E_{i'\ell}]=E_{i\ell}$ and compare
$M_{S;i,\ell}$. The only contribution in the first term is
\[
 E_{ii'}M_{S;i',\ell}=M_{S;i,\ell},
\]
since $-(p-1)=1$ in $k$. The second term contributes nothing:
an $i'$-pair of degree one cannot become saturated after one
differentiation, and the alternative has the $i$-pair low.
Hence $a_{S;i,\ell}=a_{S;i',\ell}$.

All coefficients belonging to a fixed $S$ are consequently a single
$\alpha_S$. When $r=2$ and $s=1$, each $S$ carries a single coefficient. Subtract the coboundary of
\[
 f_0=\sum_{|S|=s}\alpha_S\prod_{i\in S}m_i.
\]
Indeed $E_{i\ell}\prod_{b\in S}m_b=M_{S;i,\ell}$ if
$i\in S$, $\ell\notin S$. The other orientation and all other
short-root vectors are determined by $L$-equivariance.
For $s=0$ or $s=r$ there are no off-diagonal coefficients.
The Lie cocycle is now zero everywhere.

On every short-root group torus equivariance bounds the parameter
degree by $p-1$. The same derivative argument makes its group
cocycle zero. Root groups generate $G$, proving the assertion for
$a=0$.

Now suppose $a>0$. The differential of the action on $O_d^{(a)}$
is zero, so $dc$ is a Lie homomorphism from the perfect Lie algebra
$\mathfrak{sp}_{2r}$ to an abelian Lie algebra and is zero.
The cocycle identity then makes every differential of every
coefficient function of $c$ zero. Set $R=k[G]$ and $K=\operatorname{Frac}R$.
Over the perfect ground field, the kernel of
$d:K\to\Omega_{K/k}$ is $K^p$. A coefficient $f\in R$ is therefore
$h^p$ with $h\in K$. Since $h$ is integral over $R$ and $R$ is normal,
$h\in R$. Taking these regular $p$-th roots gives a cocycle with
coefficients in $O_d^{(a-1)}$: the representation matrices are defined
over the prime field, and injectivity of Frobenius takes the cocycle
identity to the corresponding identity one twist lower.
Repeating this step reduces to the case $a=0$. Taking powers of its coboundary
potential proves the assertion for every twist.
\end{proof}

\begin{lemma}[A symplectic socle retraction]
\label{lem:symplectic-socle-retraction}
Under the maximality hypothesis of
\cref{lem:single-weight-prefix}, the retraction
$\rho_j:C_j\to Z_M$ can be chosen $G$-equivariant as well as
$T_0$-equivariant. It remains an additive group homomorphism and
its differential is the projection onto $kX_M$.
\end{lemma}

\begin{proof}
For $j=N$ this is immediate. Otherwise the two scalar weights of
$C_j$ are $j$ and $N$. In its $T$-equivariant vector coordinates
write a point as $(v,z)$, with $v\in O_{j+2}$ and $z\in Z_M$.
The $G$-action on $v$ is its linear action, and $G$ fixes $Z_M$.
An additive cross term from $v$ to $z$ is possible only if
\[
 N=jp^a\qquad(a\ge1).
\]
If this equality fails there is no cross term, and the original
projection is already $G$-equivariant.

In the remaining case the cross term is a Frobenius-degree-$a$
linear form in $v$. The action identity gives a rational cocycle
with coefficient module $(O_{j+2}^*)^{(a)}$ (convert a right cocycle
to a left cocycle by composing its form with the inverse action).
The apolar pairing identifies $O_{j+2}^*$ with $O_{D-j-2}$.
The equality $jp^a=D-2$ makes $j=2b$ even. Reducing
$bp^a=r(p-1)-1$ modulo $p-1$ gives $b\equiv-1$, and hence
\[
 2(p-1)\mid(j+2),\qquad 2(p-1)\mid(D-j-2).
\]
The preceding lemma kills the cocycle.
Thus an additive change of coordinates
\[
 (v,z)\longmapsto(v,z-\ell(v)^{p^a})
\]
removes the cross term. The potential can be chosen of diagonal
torus weight zero: the cocycle already vanishes on that torus,
and the preceding proof uses balanced potentials.
Scalar equivariance of the change follows from $jp^a=N$.
In the new vector-group coordinates the projection to $z$ is the
required retraction. The coordinate change is additive, fixes $Z_M$ pointwise,
and has identity differential because $a\ge1$.
\end{proof}

\begin{lemma}[Two low-degree symplectic targets]
\label{lem:low-symplectic-targets}
For $k=1,3$,
\[
 \operatorname{Hom}_G(O_d,\operatorname{Sym}^kV)=0
 \quad\text{unless }d=k\text{ or }d=D-k.
\]
\end{lemma}

\begin{proof}
The apolar pairing identifies $O_d^*$ with $O_{D-d}$.
Since $k!$ is invertible, $\operatorname{Sym}^kV$ is self-dual
and is generated by the $G$-orbit of its highest pure power
$q_1^k$: pure powers span the symmetric power.
A nonzero map dual to a map in the statement must therefore take
$q_1^k$ to a nonzero highest-weight vector in $O_{D-d}$.

A monomial of weight $k\varepsilon_1$ has the form
\[
 q_1^{a_1+k}p_1^{a_1}\prod_{i>1}(q_i p_i)^{a_i}.
\]
The kernels of the positive long roots $q_i\partial_{p_i}$ force
$a_i\in\{0,p-1\}$ for $i>1$, and
$a_1\in\{0,p-1-k\}$.  Different monomials have different nonzero
images, so there is no cancellation.
For $a_1=0$, the positive root
$q_1\partial_{p_i}+q_i\partial_{p_1}$ forces every $a_i=0$.
For $a_1=p-1-k$, the same root forces every $a_i=p-1$.
The two possible nonzero images cannot cancel across branches:
their $i$-pair has respectively the high exponents
$(p-1,p-2)$ and the low exponents $(1,0)$.
Thus the only possible degrees are $k$ and $D-k$.
\end{proof}

\begin{proposition}[Intersection with \(\Gamma_3\)]
\label{prop:gamma3-socle}
\[
 \bigl(Z_M\cap\Gamma_3\bigr)_{\red}=1.
\]
\end{proposition}

\begin{proof}
Suppose that $j\ge3$ is maximal with $Z_M\subseteq\Gamma_j$.
For odd $j$ the scalar-weight equation for every component of
$\Phi_j$ is impossible: it is a sum of an odd number of odd
integers, whereas $N$ is even.

Let $j\ge4$ be even. Let
$\beta:C_{j-1}\times A_H\to C_j$ be the biadditive map induced by
$[\gamma,u]$. Choose the retraction of
\cref{lem:symplectic-socle-retraction} and put
\[
 B=\rho_j\circ\beta:C_{j-1}\times A_H\longrightarrow Z_M.
\]
For a tangent $Y\in\Lie\Gamma_{j-1}\subseteq\mathfrak h$,
differentiation in the first variable at the identity gives
$(1-\Ad u)Y\in I_U\mathfrak h\subseteq\mathfrak h_0$.
For $Z\in\Lie U\subseteq\mathfrak s$, differentiation in the
second variable gives $(\Ad\gamma-1)Z$. Since
$\gamma\in\Gamma_{j-1}\subseteq\Gamma_2$, \cref{lem:fox-depth} gives
$\Ad\gamma-1\in I_U^2$, so this tangent lies in
$I_U^2\mathfrak s\subseteq\mathfrak h_0$.
The inclusions are those of \cref{lem:two-augmentation}.
The differential of the new retraction is the same projection
onto weight $N$, and $\mathfrak h_0$ has no such weight.
Thus every component linear in either variable is zero.
A nonzero component therefore has parameter exponents $c,e\ge1$.

By \cref{lem:single-weight-prefix} its first source is
$O_{j+1}$, of scalar weight $w=j-1$, which is odd.
Every symplectic weight $\mu$ in this source is nonzero and has
coordinates in $[-(p-1),p-1]$: a zero weight would require equal
exponents in every pair and thus even potential degree, whereas
$j+1$ is odd.
For a $Q$ last variable the symplectic weights have coordinates
in $[-3,3]$ and are nonzero; for $F=V^{(1)}$ the effective twist
of its parameter degree $e$ is $e+1$.
The zero-target-weight condition and these coordinate bounds force
\[
 c=e\quad(Q),\qquad c=e+1\quad(F).
\]
Indeed an unequal effective twist would make a nonzero integer
weight coordinate at least $p$ times another nonzero coordinate,
outside the stated bounds.

In the $Q$ case, taking the common Frobenius root gives a
$G$-equivariant bilinear pairing
$O_{j+1}\times\operatorname{Sym}^3V\to k$.
Its existence requires $j+1=3$ or $D-3$ by \cref{lem:low-symplectic-targets}.
The first case is $j=2$. In the second $j=N-2$, whereas scalar
equivariance says
\[
 p^c((j-1)+1)=p^c j=N.
\]
This is impossible since $c\ge1$ and $p(N-2)>N$.

In the $F$ case, its parameter-degree-$e$ source is
$V^{(e+1)}=V^{(c)}$.
The same common-twist removal gives a bilinear pairing
$O_{j+1}\times V\to k$.
It requires $j+1=1$ or $D-1$, hence $j=0$ or $j=N$.
The scalar identity in the latter case is
\[
 p^c(j-1)+p^{c-1}(p-2)
   =p^{c-1}(pj-2)=N,
\]
which is impossible for $j=N$ and $c\ge2$.
The central summand $T_M$ contributes nothing.
Consequently $B=0$ and
$\Phi_j=B(w_{j-1},x_j)=0$, contradicting the fact that the image
of $w_j$ generates $C_j$ and $\rho_j$ is a retraction.
This proves $(Z_M\cap\Gamma_3)_{\mathrm{red}}=1$.
\end{proof}

\subsection{The two-step quotient}

\begin{lemma}[Two-variable Frobenius constraints]
\label{lem:two-variable-weights}
Let
\[
 \Psi:X\times Y\longrightarrow Z_M,
 \qquad X,Y\in\{Q,F\},
\]
be a nonzero zero-\(T_{\mathrm{sp}}\)-weight Frobenius-homogeneous component,
with parameter exponents \(e,f\ge1\).  Then \(X=Y=F\), the two parameter
exponents are equal, and
\begin{equation}
\label{eq:flux-weight}
 (p-2)p^{e-1}=h(p-1)-1.
\end{equation}
\end{lemma}

\begin{proof}
Put \(\delta_Q=0\) and \(\delta_F=1\).  The effective symplectic twist of
an \(X\)-factor with parameter exponent \(e\) is
\[
 u=e+\delta_X,
\]
because \(Q^{(e)}\) has twist \(e\), whereas
\(F^{(e)}=V^{(e+1)}\).  The weights of \(Q\) have coordinates in
\([-3,3]\), those of \(V\) have coordinates in \([-1,1]\), and neither
module has weight zero.  If the two effective twists were \(u>v\), a
zero-weight pair would give untwisted weights \(\lambda,\mu\) with
\[
 \mu=-p^{u-v}\lambda.
\]
Some coordinate of the right-hand side has absolute value at least
\(p\ge5\), whereas every coordinate of the left-hand side has absolute
value at most \(3\).  Hence the effective twists are equal.

Write
\[
 H=h(p-1)-1,
 \qquad N=2pH.
\]
For two \(Q\)-factors the common effective twist is \(u=e=f\), and scalar
equivariance gives
\[
 2p^u=2pH,
 \qquad p^{u-1}=H.
\]
Modulo \(p-1\) this is \(1=-1\), impossible since \(p\ge5\).  For a
mixed pair, say \(Q\) and \(F\), equality of effective twists gives
\(e=u\), \(f=u-1\); since \(f\ge1\), one has \(u\ge2\).  The scalar
identity becomes
\[
 p^u+(p-2)p^{u-1}=2(p-1)p^{u-1}=2pH,
\]
so \((p-1)p^{u-2}=H\), which reduces modulo \(p-1\) to \(0=-1\).
The same calculation applies with the factors reversed.

Only the \(F\)--\(F\) case remains.  Equality of effective twists gives
\(e=f\), and scalar equivariance becomes
\eqref{eq:flux-weight}.
\end{proof}

\begin{lemma}[Embedded rank-one socle]
\label{lem:rank-one-socle-gamma3}
For each \(i\), put
\[
 m_i=q_i^{p-1}p_i^{p-1},\qquad
 Z_i=\{\exp(cX_{m_i}):c\in k\}\le U_H.
\]
Then \(Z_i\subseteq\Gamma_3\); in particular, the commutator morphism of
the two flux-shear subgroups on the \(i\)-th pair is zero in \(C_2\).
\end{lemma}

\begin{proof}
Choose \(j\ne i\), set \(u=q_ip_i\), and take
\(f=q_ju^2\), \(g=-p_ju^{p-3}\).  Then
\[
 \{f,g\}=m_i,\qquad \{f,m_i\}=\{g,m_i\}=0,
 \qquad X_f^p=X_g^p=X_{m_i}^p=0.
\]
For the \(p\)-power identities, \(X_f\) fixes \(u,q_j\) and acts on
\(q_i,p_i\) by multiplication by \(\pm2q_ju\), while the analogous
coefficients for \(X_g\) are multiples of \(p_j\); hence their \(p\)-th
powers vanish because \(q_j^p=p_j^p=0\).  The same conclusion for
\(X_{m_i}\) follows from \(X_{m_i}(u)=0\) and \(u^{2p-4}=0\).
In each case the relevant nilpotent coefficient is fixed by the derivation,
so the truncated exponential identities hold.  Thus
\(\sigma_f(s)=\exp(sX_f)\) and
\(\sigma_g(t)=\exp(tX_g)\) are one-parameter subgroups of \(U_H\), of
scalar weights \(3\) and \(2p-7\).  Their composites with
\(q_H:U_H\to A_H\) are \(T_0\)-equivariant.  Relative to
\eqref{eq:AH-weight-splitting}, \cref{lem:frobenius-decomposition} allows
a nonzero component only when a target weight in \(\{1,p-2,N\}\) is a
\(p\)-power multiple of the source weight.  Parity excludes \(N\), and
the only remaining equality is \(p=5\), target weight \(p-2\), with
Frobenius exponent zero.  Its differential is the flux class of the exact
field \(X_f\) or \(X_g\), hence vanishes.  Thus \(\sigma_f(\Ga),\sigma_g(\Ga)\subseteq\Gamma_2\).  Since
\([X_f,X_g]=X_{m_i}\) commutes with both generators,
\[
 \sigma_f(s)\sigma_g(t)\sigma_f(s)^{-1}
 =\exp\!\bigl(t(X_g+sX_{m_i})\bigr)
 =\sigma_g(t)\exp(stX_{m_i}).
\]
Hence
\([\sigma_f(s),\sigma_g(t)]=\exp(stX_{m_i})\in
[\Gamma_2,\Gamma_2]_{\alg}\subseteq\Gamma_3\), and taking \(s=1\) gives
\(Z_i\subseteq\Gamma_3\).

The two flux shears on the \(i\)-th pair have congruence depth \(p-1\),
so their commutator lies in depth \(2p-3\), the top nonzero congruence
subgroup of the embedded rank-one Hamiltonian group.  That subgroup is
\(Z_i\), proving the last assertion.
\end{proof}

\begin{proposition}[Intersection with \(\Gamma_2\)]
\label{prop:gamma2-socle}
\[
 \bigl(Z_M\cap\Gamma_2\bigr)_{\red}=1.
\]
\end{proposition}

\begin{proof}
Suppose that the intersection contains a nontrivial geometric point.  By
\cref{prop:gamma3-socle}, \(j=2\) is maximal with this property, so
\cref{lem:socle-retraction} gives a \(T\)-equivariant retraction
\(C_2\to Z_M\).  Its composite with \(w_2\) is a nonzero alternating
biadditive map
\[
 \Phi_2:A_H\times A_H\longrightarrow Z_M.
\]

Use the fixed splitting~\eqref{eq:AH-weight-splitting}.  Since the central
summand \(T_M\) does not contribute, some nonzero Frobenius-homogeneous
component of \(\Phi_2\) factors through \(X\times Y\), with
\(X,Y\in\{Q,F\}\).  If its parameter exponents are \(e,f\), the
\(j=2\) case of \eqref{eq:frobenius-weight} reads
\[
 ap^e+bp^f=2p\bigl(h(p-1)-1\bigr),
 \qquad a,b\in\{1,p-2\}.
\]
Reduction modulo \(p\) shows that \(e,f\ge1\): if exactly one exponent
were zero, the corresponding coefficient \(1\) or \(p-2\) would survive
modulo \(p\), while if both were zero then
\(a+b\in\{2,p-1,2p-4\}\), none of which is divisible by \(p\).
By \cref{lem:two-variable-weights}, the component must therefore be an
equal-exponent \(F\)--\(F\) component satisfying
\eqref{eq:flux-weight}.

Now \(F^{(e)}=V^{(e+1)}\) has weights
\[
 \pm p^{e+1}\varepsilon_i,
 \qquad 1\le i\le r.
\]
Since the target has \(T_{\mathrm{sp}}\)-weight zero, the zero-weight part
of an alternating law on \(F^{(e)}\times F^{(e)}\) is determined by one
coefficient for each pair of opposite weight lines
\[
 +p^{e+1}\varepsilon_i,
 \qquad -p^{e+1}\varepsilon_i.
\]
Fix \(i\).  These two lines are realized by the flux shears
\[
 S_i(a)^*:p_i\mapsto p_i+a q_i^{p-1},
 \qquad
 T_i(b)^*:q_i\mapsto q_i+b p_i^{p-1}.
\]
By \cref{lem:rank-one-socle-gamma3}, their commutator morphism is
zero in \(C_2\).  Hence the restriction of \(w_2\), and therefore of
\(\Phi_2\), to the corresponding pair of flux lines is zero.  This kills
the \(i\)-th coefficient, including all of its Frobenius terms.  Varying
\(i\) annihilates the zero-weight component, a contradiction.
\end{proof}

\subsection{Finite-field points and the socle line}

\begin{theorem}[The resonant case]
\label{thm:resonant-hamiltonian}
Let \(r\ge2\) and \(r+1=hp\).  Then
\[
 \bigl(Z_M\cap[U_H,U_H]_{\alg}\bigr)_{\red}=1.
\]
For every \(q=p^f\),
\[
 Z_M(\F_q)\cap[U_H(\F_q),U_H(\F_q)]=1.
\]
If
\[
 1\longrightarrow Z_M(\F_q)\longrightarrow U_H(\F_q)
 \longrightarrow U_H(\F_q)/Z_M(\F_q)\longrightarrow1
\]
is the central extension by the socle line, then its integral transgression
\[
 \tau_{M,q}:H_2(U_H(\F_q)/Z_M(\F_q);\mathbf Z)
 \longrightarrow Z_M(\F_q)
\]
is zero.
\end{theorem}

\begin{proof}
Since \(\Gamma_2=[U_H,U_H]_{\alg}\), \cref{prop:gamma2-socle} gives
\[
 \bigl(Z_M\cap[U_H,U_H]_{\alg}\bigr)_{\red}=1.
\]
Taking \(k=\F_q\), the reduced intersection is trivial over \(\F_q\).  Every \(\F_q\)-point factors through the reduction, and every
abstract commutator belongs to \([U_H,U_H]_{\alg}(\F_q)\); the finite-point
intersection is therefore trivial.

The integral five-term homology sequence of the central extension gives
\[
 \im\tau_{M,q}
 =Z_M(\F_q)\cap[U_H(\F_q),U_H(\F_q)].
\]
The right-hand side is trivial, so \(\tau_{M,q}=0\).
\end{proof}

\begin{proof}[Proof of \cref{thm:hamiltonian-main}]
The rank-one assertion and the nonresonant algebraic intersection are
\cref{cor:hamiltonian-nonresonant}.  In the nonresonant case with
\(r\ge2\), every \(\F_q\)-point of
\(Z_M\cap[U_H,U_H]_{\alg}\) factors through its reduced subscheme; hence
the triviality of the reduced intersection also gives
\[
 Z_M(\F_q)\cap[U_H(\F_q),U_H(\F_q)]=1.
\]
The resonant algebraic and finite-point assertions are
\cref{thm:resonant-hamiltonian}.
\end{proof}
\section{The Witt--Jacobson and contact families}
\label{sec:witt-contact}

\subsection{Reduced radicals}

\begin{proposition}
\label{prop:wh-radicals}
The linear-part maps give split decompositions
\[
 G_W=U_W\rtimes\GL_n,
 \qquad
 G_H=U_H\rtimes\operatorname{CSp}_{2r},
\]
and
\[
 R_u(G_W)=U_W,
 \qquad R_u(G_H)=U_H.
\]
\end{proposition}

\begin{proof}
Linear coordinate substitutions split both maps.  In the Hamiltonian
case they act on \(\omega_H\) through the defining similitude character.
An element in either kernel of the linear-part map is upper unitriangular
on the monomial basis ordered by total degree, hence lies in a unipotent
algebraic subgroup.

Scalar dilation \(\delta_t:x_i\mapsto tx_i\) normalizes both kernels and
satisfies
\[
 \delta_tg\delta_t^{-1}(x_i)
 =x_i+\sum_{\abs\alpha\ge2}
 t^{\abs\alpha-1}c_{i,\alpha}x^\alpha.
\]
The action extends to \(t=0\) with value the identity, so each kernel is
connected and unipotent.  Since \(\GL_n\) and
\(\operatorname{CSp}_{2r}\) are reductive, no larger connected normal
unipotent subgroup can occur.  Hence these kernels are the unipotent
radicals.
\end{proof}

\subsection{The Witt--Jacobson congruence group}

Assume throughout this subsection that \(n\ge2\).  Let \(U_W\) be the
augmentation-preserving automorphism group of \(O(n;1)\) with identity
linear part, and set
\[
 F^sU_W=\{g\in U_W:g(x_i)-x_i\in\fm^s\text{ for every }i\},
 \qquad s\ge2.
\]
Its degree-\(d\) tangent component is
\[
 \cW_d=O_d\otimes V.
\]

\begin{lemma}[Homogeneous Witt--Jacobson brackets]
\label{lem:witt-generation}
For \(3\le d\le D\),
\[
 [\cW_2,\cW_{d-1}]=\cW_d.
\]
\end{lemma}

\begin{proof}
It suffices to generate a monomial field
\(X=x^\alpha\partial_j\in\cW_d\).  Suppose first that
\(\alpha_s>0\) for some \(s\ne j\).  For
\[
 H_h=\sum_{\ell=1}^n h_\ell x_\ell\partial_\ell,
 \qquad Q_h=x_sH_h,
 \qquad Y=x^{\alpha-e_s}\partial_j,
\]
one has
\[
 [Q_h,Y]
 =\langle h,\alpha-e_s-e_j\rangle X.
\]
The coefficient vanishes for every \(h\in k^n\) only if
\(\alpha-e_s-e_j=0\) in \(\F_p^n\).  Because every box exponent lies in
\([0,p-1]\), this forces \(\alpha=e_s+e_j\), contrary to \(d\ge3\).
A suitable choice of \(h\) therefore yields \(X\).

If no such \(s\) exists, then \(X=x_j^d\partial_j\), with \(d\le p-1\).
Choose \(s\ne j\).  Since
\[
 [x_j^2\partial_s,x_sx_j^{d-2}\partial_j]
 =x_j^d\partial_j-2x_sx_j^{d-1}\partial_s,
\]
the second summand is covered by the first case, and hence so is
\(x_j^d\partial_j\).
\end{proof}

\begin{lemma}[Balanced Witt--Jacobson brackets]
\label{lem:witt-balanced}
If \(a,b\ge2\) and \(a+b-1\le D\), then
\[
 [\cW_a,\cW_b]=\cW_{a+b-1}.
\]
\end{lemma}

\begin{proof}
The inclusion from left to right follows from coefficient degree.  For the
reverse inclusion, assume \(a\le b\) and take a monomial field
\[
 X=x^\alpha\partial_j\in\cW_{a+b-1}.
\]
Throughout we use
\[
 [x^u\partial_i,x^v\partial_j]
 =v_i x^{u+v-e_i}\partial_j-u_jx^{u+v-e_j}\partial_i,
\]
a monomial being understood to vanish once one of its exponents exceeds
\(p-1\).

Suppose first that \(\alpha_j\le b\).  Then \(|\alpha|-\alpha_j\ge a-1\), so
\(x^\alpha\) has a monomial divisor \(x^u\) supported away from \(x_j\) with
\(|u|=a-1\).  Put
\[
 H_h=\sum_{\ell=1}^n h_\ell x_\ell\partial_\ell,\qquad
 Q=x^uH_h,\qquad Y=x^{\alpha-u}\partial_j,
\]
so that \(Q\in\cW_a\) and \(Y\in\cW_b\).  Since \(u_j=0\), the \(\ell\)-th
summand of \(Q\) contributes \((\alpha-u-e_j)_\ell X\) to \([Q,Y]\).
Truncation to the height-one box removes only summands with coefficient
zero: if \(x^{u+e_\ell}\) vanishes, then \(u_\ell=p-1\), hence \(\alpha_\ell=p-1\) and
\((\alpha-u-e_j)_\ell=0\).  Therefore
\[
 [Q,Y]
 =\bigl\langle h,\alpha-u-e_j\bigr\rangle X.
\]
The vector \(\alpha-u-e_j\) is nonzero in \(\F_p^n\): its coordinates away
from \(j\) lie in \([0,p-1]\) and its \(j\)-th coordinate in \([-1,p-2]\), so
it can vanish only for \(\alpha-u=e_j\), against \(b\ge2\).  A suitable
choice of \(h\) now produces \(X\).

Now suppose that \(\alpha_j>b\).  Put
\[
 R=a+b-1-\alpha_j,\qquad
 \beta=\alpha-\alpha_je_j,
\]
so that \(0\le R\le a-2\) and \(|\beta|=R\).  For a monomial divisor
\(x^u\mid x^\beta\) of degree \(t\), define
\[
 A_u=x^ux_j^{a-t}\partial_j,\qquad
 B_u=x^{\beta-u}x_j^{b-R+t}\partial_j.
\]
These lie in \(\cW_a\) and \(\cW_b\), and neither vanishes: the two
exponents of \(x_j\) sum to \(a+b-R=\alpha_j+1\le p\), while \(a-t\ge2\)
because \(t\le R\le a-2\), and \(b-R+t=\alpha_j+1-a+t\ge2\) because
\(\alpha_j>b\ge a\); hence neither exponent exceeds \(p-2\).  The bracket
formula gives
\[
 [A_u,B_u]=(b-a-R+2t)X.
\]
If \(R\ge1\), the values \(t=0,1\) are both available and the two
coefficients differ by \(2\), so one of them is nonzero.  If \(R=0\) and
\(a\ne b\), the single coefficient \(b-a\) is nonzero, since
\(0<b-a<a+b=\alpha_j+1\le p\).

There remains the case \(a=b=s\) with \(R=0\), that is
\(X=x_j^{2s-1}\partial_j\).  Choose \(i\ne j\), which is possible because
\(n\ge2\).  Then
\[
 [x_j^s\partial_i,x_ix_j^{s-1}\partial_j]
 =X-sx_ix_j^{2s-2}\partial_i,
\]
both entries lying in \(\cW_s\) because \(s\le2s-1=\alpha_j\le p-1\).  The
field \(x_ix_j^{2s-2}\partial_i\) carries the exponent \(1\le s\) in its
derivation coordinate and so falls under the first case.  Hence \(X\) is a
bracket as well.
\end{proof}

\begin{proposition}[Witt--Jacobson commutators]
\label{prop:witt-series}
The quadratic-jet map
\[
 j_2(g)=\sum_i\bigl(g(x_i)-x_i\bigr)_{[2]}\partial_i
\]
is a surjective homomorphism \(U_W\to\cW_2\), with kernel \(F^3U_W\), and
\[
 U_W/[U_W,U_W]_{\alg}\simeq\cW_2.
\]
More generally, for \(s\ge2\) and every \(q=p^f\),
\[
 [F^sU_W,U_W]_{\alg}=F^{s+1}U_W,
 \qquad
 [F^sU_W(\F_q),U_W(\F_q)]=F^{s+1}U_W(\F_q).
\]
Moreover,
\[
 [F^sU_W,F^sU_W]_{\alg}=F^{2s-1}U_W,
 \qquad
 [F^sU_W(\F_q),F^sU_W(\F_q)]=F^{2s-1}U_W(\F_q).
\]
Consequently
\[
 [U_W,U_W]_{\alg}=F^3U_W,
 \qquad \gamma_j(U_W)=F^{j+1}U_W,
 \qquad U_W^{(j)}=F^{2^j+1}U_W=\gamma_{2^j}(U_W),
\]
and, with \(D=n(p-1)\),
\[
 \operatorname{cl}(U_W)=D-1,
 \qquad
 \operatorname{dl}(U_W)=\left\lceil\log_2D\right\rceil.
\]
The same lower-central and derived-series formulas hold for
\(U_W(\F_q)\), and
\[
 \Phi(U_W(\F_q))=F^3U_W(\F_q),
 \qquad P_j(U_W(\F_q))=F^{j+1}U_W(\F_q).
\]
\end{proposition}

\begin{proof}
Every \(X=\sum_i f_i\partial_i\in\cW_d\) has a homogeneous lift in
\(F^dU_W\): since every element of \(\fm\) has \(p\)-th power zero, the
substitution \(x_i\mapsto x_i+f_i\) defines an endomorphism of \(O(n;1)\),
and its identity action on \(\fm/\fm^2\) makes it an automorphism by
Nakayama's lemma.  Therefore
\[
 F^dU_W/F^{d+1}U_W\simeq\cW_d.
\]
Substitution expansion gives
\[
 [F^aU_W,F^bU_W]\subseteq F^{a+b-1}U_W,
\]
with Lie bracket on the first nonzero quotient.  Hence
\cref{lem:witt-generation} gives, for \(2\le s<D\),
\[
 \gr^{s+1}[F^sU_W,U_W]_{\alg}=\cW_{s+1}.
\]
Equivalently, any \(g\in F^{s+1}U_W\) can be multiplied by a product of
commutators from \([F^sU_W,U_W]\) to raise its depth by one.  Iteration
through the finite filtration proves
\([F^sU_W,U_W]_{\alg}=F^{s+1}U_W\); for \(s\ge D\) both sides are trivial.
Each depth-raising commutator is obtained by lifting an identity in the
homogeneous Witt--Jacobson algebra and may therefore be chosen over \(\F_q\).  The
same induction gives the finite-point equality.

Fix \(s\ge2\).  The coefficient-degree estimate gives
\([F^sU_W,F^sU_W]_{\alg}\subseteq F^{2s-1}U_W\).  For every
\(d\ge2s-1\), \cref{lem:witt-balanced}, applied with
\(a=s\) and \(b=d-s+1\), gives
\[
 \cW_d=[\cW_s,\cW_{d-s+1}].
\]
Thus every leading field in depth \(d\) is the leading field of a product
of commutators of elements whose depths are at least \(s\).  Multiplying by
such a product raises the depth, and iteration through the finite
filtration proves
\[
 [F^sU_W,F^sU_W]_{\alg}=F^{2s-1}U_W.
\]
All monomial identities in the proof of \cref{lem:witt-balanced} have
coefficients in the prime field.  Scaling one factor realizes arbitrary
\(\F_q\)-coefficients, so the same depth-raising argument proves the
finite-point equality.  If \(2s-1>D\), both sides are trivial.

Taking \(s=2\) gives \([U_W,U_W]_{\alg}=F^3U_W\).  Quadratic terms add under composition, and homogeneous lifting makes
\(j_2\) surjective.  Hence \(\ker j_2=F^3U_W\), giving the stated
abelianization; the lower-central formula follows by induction.  The
self-commutator equality gives, again by induction,
\[
 U_W^{(j)}=F^{2^j+1}U_W=\gamma_{2^j}(U_W),
\]
over both \(k\) and \(\F_q\).  Since \(F^DU_W\ne1\) and
\(F^{D+1}U_W=1\), the class is \(D-1\), and the least \(j\) for which
\(2^j+1>D\) is \(\lceil\log_2D\rceil\).
For the finite-point series,
\[
 (F^sU_W(\F_q))^p
 \subseteq F^{1+p(s-1)}U_W(\F_q)
 \subseteq F^{s+1}U_W(\F_q).
\]
Hence \(U_W(\F_q)^p\subseteq F^3U_W(\F_q)\), and therefore
\(\Phi(U_W(\F_q))=F^3U_W(\F_q)\).  Since \(P_1=F^2U_W(\F_q)\), the
induction hypothesis \(P_j=F^{j+1}U_W(\F_q)\) gives
\[
 P_{j+1}=P_j^p[P_j,U_W(\F_q)]=F^{j+2}U_W(\F_q),
\]
completing the induction.
\end{proof}

\subsection{Distribution comparison}

For an affine group scheme \(\mathbf G\) of finite type, write
\(\Lie_j(\mathbf G)\) for the primitive distributions admitting a
divided-power sequence through level \(p^j\).  Thus
\(\Lie_0(\mathbf G)=\Lie(\mathbf G)\), and these spaces decrease with \(j\).

\begin{lemma}[Distribution comparison criterion]
\label{lem:scheme-comparison}
Let \(\varphi:\mathbf G\to\mathbf H\) be a morphism of affine group
schemes of finite type over a perfect field.  Suppose that
\(\varphi_{\red}:\mathbf G_{\red}\to\mathbf H_{\red}\) is an
isomorphism and that the induced map on distributions restricts to
isomorphisms
\[
 \Lie_j(\mathbf G)\xrightarrow{\sim}\Lie_j(\mathbf H)
 \qquad(j\ge0).
\]
Then \(\varphi\) is an isomorphism.
\end{lemma}

\begin{proof}
Write \(D_G=\Dist(\mathbf G)\) and \(D_H=\Dist(\mathbf H)\).  The
distribution criterion recalled in \cite[Section~3]{BahturinKochetov2011}
shows that the reduced-group isomorphism together with the tangent-space
isomorphism makes \(\Dist(\varphi):D_G\to D_H\) injective.  Identify
\(D_G\) with its image.  For either distribution Hopf algebra,
\[
 \Lie_j(D)=V^j(D)\cap\operatorname{Prim}(D),
\]
and a primitive lies in \(\Lie_j(D)\) exactly when it extends to a
divided-power sequence of length \(p^j\)
\cite[Theorem~2]{Sweedler1967}.  Hence the hypotheses give
\(\Lie_j(D_G)=\Lie_j(D_H)\) for every \(j\).  Dieudonn\'e's
Hopf-subalgebra criterion
\cite[II, \S3, no.~2, Corollary~1]{Dieudonne1973} therefore yields
\[
 D_G=D_H.
\]
Thus \(\varphi\) induces an isomorphism on the formal completions at the
identity.  Translation by the common reduced subgroup gives the same
statement at every geometric point; the formal criterion for finite-type
schemes then implies that \(\varphi\) is an isomorphism.
\end{proof}

\subsection{The full contact automorphism scheme}

Put
\[
 O_K=O(2r+1;1)
 =k[q_1,\ldots,q_r,p_1,\ldots,p_r,z]/
 (q_1^p,\ldots,q_r^p,p_1^p,\ldots,p_r^p,z^p)
\]
and
\[
 \alpha_K=dz+\sum_{i=1}^r(q_i\,dp_i-p_i\,dq_i).
\]
Let
\[
 \mathbf C_K
 =\operatorname{Stab}_{\mathbf{Aut}(O_K)}(O_K^\times\alpha_K).
\]
Write \(K'\) for the full contact algebra and
\(L=[K',K']\) for its simple restricted derived algebra.  Then
\(L=K'\) unless \(p\mid r+2\), in which case \(L\) has codimension one
in \(K'\).  In contact degree,
\[
 K'_{-2}=k1,
 \qquad K'_{-1}=\bigoplus_{i=1}^r(kq_i\oplus kp_i),
 \qquad K'_{\ge0}=K'_{(0)}.
\]
Set \(\mathfrak c_{K,(0)}=K'\cap W(2r+1;1)_{(0)}\).

Faithfully flat descent reduces the contact comparison to the case in which
\(k\) is algebraically closed.  The contact automorphism and derivation
theorems give \cite[Theorems~7.3.2 and~7.1.2]{Strade2004}
\[
 \Ad:(\mathbf C_K)_{\red}\xrightarrow{\sim}
       \mathbf{Aut}(L)_{\red},
 \qquad
 \ad:K'\xrightarrow{\sim}\Der(L).
\]

\begin{lemma}[The length-\(p\) squaring obstruction]
\label{lem:length-p-obstruction}
Let \(A\) be a finite-dimensional Lie algebra in characteristic \(p>2\).
If \(D\in\Der(A)\cap\Lie_1(\mathbf{Aut}(A))\), then
\([\Sq(D)]=0\) in \(H^2(A,A)\).  Moreover,
\(D\mapsto[\Sq(D)]\) is Frobenius-semilinear on \(\Der(A)\).  For
\(D=\ad y\), write \([\Sq(y)]\).
\end{lemma}

\begin{proof}
Let \(\Delta(t)=\sum_{i=0}^p t^i\delta_i\) be a length-\(p\)
divided-power sequence with \(\delta_1=D\).  Once the terms of order below
\(m<p\) are normalized, \(\delta_m-D^m/m!\) is primitive, hence a
derivation; composing with the truncation of its exponential removes that
term without changing lower orders.  Sweedler's normalization
\cite[Lemma~7]{Sweedler1967} therefore gives \(\delta_i=D^i/i!\) for
\(i<p\), and the order-\(p\) identity becomes
\[
 \delta_p([a,b])-[\delta_p(a),b]-[a,\delta_p(b)]
 =\sum_{i=1}^{p-1}\frac{[D^i(a),D^{p-i}(b)]}{i!(p-i)!}=\Sq(D).
\]
The left side is \(-d_{\mathrm{CE}}\delta_p\), proving \([\Sq(D)]=0\).
Scalar semilinearity follows directly.  For additivity, the truncated Hasse
exponential \(E_D(t)=\sum_{i<p}t^iD^i/i!\) has order-\(p\) bracket defect
\(\Sq(D)\).  Hence \(E_D(t)E_E(t)\) has defect
\(\Sq(D)+\Sq(E)\) and first coefficient \(D+E\).  Normalizing its lower
coefficients changes the order-\(p\) defect only by a
Chevalley--Eilenberg coboundary, so
\([\Sq(D+E)]=[\Sq(D)]+[\Sq(E)]\).
\end{proof}

\begin{proposition}[Positive contact Verschiebung spaces]
\label{prop:contact-verschiebung}
For every \(j\ge1\),
\[
 \Lie_j(\mathbf C_K)=\mathfrak c_{K,(0)},\qquad
 \Lie_j(\mathbf{Aut}(L))=\ad\mathfrak c_{K,(0)},
\]
and \(\Ad\) induces an isomorphism between these spaces.
\end{proposition}

\begin{proof}
By the divided-power calculation of Allen and Sweedler for the coordinate
automorphism scheme
\cite[Lemma~3.5]{AllenSweedler1969},
\(\Lie_j(\mathbf{Aut}(O_K))=W(2r+1;1)_{(0)}\) for \(j\ge1\).  Hence
\(\Lie_j(\mathbf C_K)\subseteq\mathfrak c_{K,(0)}\).  Conversely,
\((\mathbf C_K)_{\red}\) is smooth with tangent algebra
\(\mathfrak c_{K,(0)}\), so every tangent vector admits divided-power sequences of arbitrary length.  Hence
\(\Lie_j(\mathbf C_K)=\mathfrak c_{K,(0)}\).  The map
\(\Dist(\Ad)\) preserves primitives and commutes with Verschiebung, so
\(\ad\mathfrak c_{K,(0)}\subseteq\Lie_j(\mathbf{Aut}(L))\).

It suffices to prove the reverse inclusion for \(j=1\).  Viviani's computation
\cite[Theorem~1.1]{Viviani2009} gives
\[
 H^2(L,L)=k[\Sq(1)]\oplus
 \bigoplus_{i=1}^r\bigl(k[\Sq(q_i)]\oplus k[\Sq(p_i)]\bigr),
\]
and the displayed classes are linearly independent.  Hence for
\(0\ne y=a1+\sum_i(b_iq_i+c_ip_i)\in K'_{-2}\oplus K'_{-1}\),
\[
 [\Sq(y)]=a^p[\Sq(1)]+
 \sum_i\bigl(b_i^p[\Sq(q_i)]+c_i^p[\Sq(p_i)]\bigr)\ne0.
\]
Now take \(\xi\in\Lie_1(\mathbf{Aut}(L))\).  Since
\(\ad:K'\xrightarrow{\sim}\Der(L)\), write uniquely
\(\xi=\ad(y+D_0)\) with \(y\) as above and
\(D_0\in\mathfrak c_{K,(0)}\).  The inclusion already proved gives
\(\ad D_0\in\Lie_1(\mathbf{Aut}(L))\), hence also
\(\ad y\in\Lie_1(\mathbf{Aut}(L))\).  By \cref{lem:length-p-obstruction}, \([\Sq(y)]=0\), forcing \(y=0\).  Hence
\(\Lie_1(\mathbf{Aut}(L))=\ad\mathfrak c_{K,(0)}\).  Since the
Verschiebung spaces decrease with \(j\),
\[
 \ad\mathfrak c_{K,(0)}\subseteq\Lie_j(\mathbf{Aut}(L))
 \subseteq\Lie_1(\mathbf{Aut}(L))=\ad\mathfrak c_{K,(0)},
\]
and the proposition follows.
\end{proof}

\begin{proof}[Proof of \cref{thm:witt-contact-main}]
The radical assertions are \cref{prop:wh-radicals}, and the Witt--Jacobson series is
\cref{prop:witt-series}.  For the contact algebra, the contact automorphism
theorem identifies the reduced subgroup schemes under \(\Ad\), while the
derivation theorem identifies the tangent map
\(\Lie(\mathbf C_K)=K'\xrightarrow{\sim}\Der(L)\).  By
\cref{prop:contact-verschiebung}, \(\Ad\) also induces isomorphisms on
\(\Lie_j\) for every \(j\ge1\).  The comparison criterion
\cref{lem:scheme-comparison} then gives
\(\Ad:\mathbf C_K\xrightarrow{\sim}\mathbf{Aut}(L)\).
\end{proof}
 
\bibliographystyle{amsplain}
\section*{Acknowledgements}

I am grateful to Serge Skryabin for his helpful corrections and
bibliographical suggestions.  His comments also prompted me to work out the
Hamiltonian lower central and derived series.

\end{document}